\documentclass[aihp]{imsart}

\RequirePackage[OT1]{fontenc}
\RequirePackage{amsthm,amsmath}
\usepackage{amsfonts}
\usepackage{bbm}
\usepackage{graphicx, subcaption}

\RequirePackage[numbers]{natbib}
\RequirePackage[english]{babel}

\usepackage{hyperref}

\arxiv{1905.01691}

\startlocaldefs
\newcommand{\C}{\mathbb C}
\newcommand{\E}{\mathbb{E}}
\newcommand{\N}{\mathbb{N}}
\newcommand{\Z}{\mathbb{Z}}

\newcommand{\R}{\mathbb{R}}

\newcommand{\1}{\mathbbm{1}}

\renewcommand{\Im}{\operatorname{Im}}
\renewcommand{\Re}{\operatorname{Re}}
\renewcommand{\i}{\mathrm{i}}
\newcommand{\erf}{\operatorname{erf}}
\newcommand{\erfc}{\operatorname{erfc}}
\newcommand{\refr}{\mathrm{refr}}
\newcommand{\loc}{\mathrm{loc}}

\numberwithin{equation}{section}
\newtheorem{corollary}{Corollary}[section]
\newtheorem{lemma}[corollary]{Lemma}
\newtheorem{proposition}[corollary]{Proposition}
\newtheorem{theorem}[corollary]{Theorem}
\newtheorem{taggedassumptionx}{Assumption}
\newenvironment{assumption}[1]
{\renewcommand\thetaggedassumptionx{#1}\taggedassumptionx}
{\endtaggedassumptionx}

\theoremstyle{definition}

\theoremstyle{remark}
\newtheorem{example}[corollary]{Example}
\newtheorem{remark}[corollary]{Remark}
\endlocaldefs

\begin{document}
\begin{frontmatter}

\title{Spectral analysis of the zigzag process}

\runtitle{Spectral analysis of the zigzag process}

\begin{aug}

\author{\fnms{Joris} \snm{Bierkens}\thanksref{a}
\corref{}
\ead[label=e1]{joris.bierkens@tudelft.nl}
}
\and
\author{\fnms{Sjoerd M.} \snm{Verduyn Lunel}\thanksref{b}\ead[label=e2]{S.M.VerduynLunel@uu.nl}}

\address[a]{Delft University of Technology \\ Delft Institute of Applied Mathematics \\ \printead{e1}}

\address[b]{Utrecht University \\ Department of Mathematics \\ \printead{e2}}

\thankstext{a}{Joris Bierkens acknowledges funding by the research programme `Zig-zagging through computational barriers' with project number 016.Vidi.189.043, which is financed by the Netherlands Organisation for Scientific Research (NWO).}

\runauthor{J. Bierkens, S. M. Verduyn Lunel}

\affiliation{Delft Institute of Applied Mathematics and Utrecht University}

\end{aug}

\begin{abstract}
		The zigzag process is a variant of the telegraph process with position dependent switching intensities. A characterization of the $L^2$-spectrum for the generator of the one-dimensional zigzag process is obtained in the case where the marginal stationary distribution on $\R$ is unimodal and the refreshment intensity is zero. Sufficient conditions are obtained for a spectral mapping theorem, mapping the spectrum of the generator to the spectrum of the corresponding Markov semigroup. 
	Furthermore results are obtained for symmetric stationary distributions and for perturbations of the spectrum, in particular for the case of a non-zero refreshment intensity. In the examples we consider (including a Gaussian target distribution) a slight increase of the refreshment intensity above zero results in a larger $L^2$-spectral gap, corresponding to an improved convergence in $L^2$.
\end{abstract}

\begin{abstract}[language=french]
Le processus de ZigZag est une variante du processus du t\'{e}l\'{e}graphe avec des intensit\'{e}s de retournement d\'{e}pendant de la position.  Nous obtenons une caract\'{e}risation du spectre en norme $L^2$ du g\'{e}n\'{e}rateur du processus de zigzag unidimensionnelle, dans le cas o\`{u} la distribution marginale stationnaire sur ${\mathbb R}$ est unimodale, avec une fr\'{e}quence de r\'{e}\'{e}chantillonage nulle. Nous obtenons des conditions suffisantes pour un th\'{e}or\`{e}me d'isomorphisme spectral, identifiant le spectre du g\'{e}n\'{e}rateur \`{a} celui du semi-groupe de Markov correspondant. Par ailleurs, nous obtenons des r\'{e}sultats pour les distributions stationnaires sym\'{e}triques ainsi que pour les perturbations du spectre dans le cas particulier d'une fr\'{e}quence de r\'{e}\'{e}chantillonage non nulle. Enfin, nous consid\'{e}rons dans les exemples (avec une distribution cible gaussienne) un faible taux de rafraichissement positif par rapport aux r\'{e}sultats du taux nul, ce qui induit une plus grande bande dans le spectre $L^2$, correspondant \`{a} une convergence am\'{e}lior\'{e}e en norme $L^2$.
\end{abstract}

\begin{keyword}[class=MSC2020]
\kwd[Primary ]{37A30}
\kwd{47A10} \kwd[; secondary ]{60J25}
\end{keyword}

\begin{keyword}
\kwd{spectral theory}
\kwd{non-reversible Markov process}
\kwd{Markov semigroup}
\kwd{zigzag process}
\kwd{exponential ergodicity}
\kwd{perturbation theory}
\kwd{telegraph process}
\kwd{piecewise deterministic Markov process}
\end{keyword}



\end{frontmatter}

\section{Introduction}

In recent years piecewise deterministic Markov processes have emerged as a useful computational tool in stochastic simulation, for example for Bayesian statistics or statistical physics. A particular instance of such a process is the \emph{zigzag process}, described already in e.g. \cite{Fontbona2012,Monmarche2016} and given its name in \cite{BierkensRoberts2017}. The zigzag process is a variant of, and extends the \emph{telegraph process} \cite{Kac1951} and is intimately related to the \emph{Bouncy Particle Sampler} \cite{PetersDeWith2012,BouchardCoteVollmerDoucet2017}.

The zigzag process can be defined in multiple dimensions \cite{BierkensFearnheadRoberts2016} but here we consider only one spatial dimension. In this setting the zigzag process is a Markov process $(X(t), \Theta(t))$ in the state space $E := \R\times \{-1,+1\}$. Conditional on the \emph{velocity process} $(\Theta(t))_{t \geq 0}$, the \emph{position process} $(X(t))_{t \geq 0}$ in $\R$ is completely determined by the relation $X(t) = X(0) + \int_0^t \Theta(s) \, ds$. The velocity process $(\Theta(t))_{t \geq 0}$ in $\{-1,+1\}$ switches sign at inhomogeneous Poisson rate $\lambda(X(t),\Theta(t))$, where $\lambda : E \rightarrow [0,\infty)$ is a continuous function. Thus the \emph{switching intensity} function $\lambda$ has a crucial impact on the dynamics exhibited by the zigzag process.

For a given absolutely continuous \emph{potential function} $U : \R\rightarrow \R$, the switching intensity $\lambda$ can be chosen as $\lambda(x,\theta) = (\theta U'(x) \vee 0) + \lambda_{\refr}(x)$ for some non-negative \emph{refreshment intensity} function $\lambda_{\refr}:\R\rightarrow [0,\infty)$. This condition results in the zigzag process having stationary distribution $\mu$ on $E$ with marginal density on $\R$ proportional to $e^{-U}$. This observation makes the zigzag process into a useful computational device for the stochastic simulation of a given target density proportional to $e^{-U}$.

In the design of stochastic simulation algorithms it is vital to understand the convergence to stationarity, i.e., the quality of the approximation of the probability density proportional to $e^{-U}$ by the law of $X(t)$. For the zigzag process this question has already been addressed in several ways. In \cite{Fontbona2015,BierkensRoberts2017} a Lyapunov function argument has been used to show exponentially fast convergence to stationarity in total variation norm. In \cite{BierkensDuncan2016} a central limit theorem for $\frac 1 {\sqrt{T}} \int_0^T (f(X(s), \Theta(s))  - \pi(f)) \, ds$, for $T \rightarrow \infty$, was obtained under suitable conditions. Here $\pi$ denotes the normalization of $\mu$ into a probability distribution. The hypocoercivity approach of \cite{Dolbeault2015} is applied in \cite{Andrieu2018} to the zigzag process, amongst other piecewise deterministic processes, in order to obtain exponential convergence in $L^2(\mu)$. Recently \cite{Lu2020} provided estimates on the exponential convergence of PDMPs using a time-dependent Poincar\'e inequality, that seem to result in a good understanding of dependence on, e.g., dimensionality, anisotropy and convexity. However, typically these results either do not obtain a tractable  quantitative bound on the rate of convergence, and/or provide only a one-sided estimate on the rate of convergence.

In order to obtain quantitative result on the rate of convergence to stationarity, it is natural to investigate the spectral properties of the Markov semigroup $(P(t))_{t \geq 0}$ connected to the zigzag process, and this is the aim of this work. The probabilistic interpretation of $P(t)$ is that
\[ P(t) f(x,\theta) = \E_{x,\theta} f(X(t),\Theta(t)), \quad t \geq 0,\]
where $\E_{x,\theta}$ denotes the expectation over trajectories of the Markov process $(X(t),\Theta(t))$ with initial position $(X(0), \Theta(0)) = (x,\theta)$. In this paper we take a fully analytical approach, and thus construct and analyse the semigroup $(P(t))_{t \geq 0}$ using functional analytical arguments with no reliance on probabilistic results.

In a similar vein in \cite{MicloMonmarche2013} a rather complete understanding of the spectral properties in $L^2(\mu)$ is obtained for a periodic variant of the telegraph process, with uniform marginal stationary distribution on the position space. As in \cite{Andrieu2018, MicloMonmarche2013} we consider the spectrum of the zigzag semigroup in the Hilbert space $L^2(\mu)$. Although the zigzag semigroup is not selfadjoint, it still has remarkable symmetry properties as a semigroup in $L^2(\mu)$. In particular we obtain an  elegant characterization of the adjoint of the generator (see Lemma~\ref{lem:adjoint}), and a decomposition of the spectrum in the case of a symmetric stationary distribution (Section~\ref{sec:symmetric}).

Our results rely for a significant part on the explicit solution of the eigenvalue and resolvent equations, which is possible if $\lambda_{\refr} = 0$. We anticipate that the results can be extended to $\lambda_{\refr} \neq 0$ and to a wider variety of function spaces, using an approach based upon characteristic matrices \cite{Kaashoek1992}, which will be the focus of further research.

The structure of this work is as follows. First, in Section~\ref{sec:notation-assumptions} we introduce notation and our main assumptions which are used throughout this document. In Section~\ref{sec:semigroup} we construct, imposing minimal assumptions and taking a purely analytical approach, the zigzag semigroup $(P(t))_{t \geq 0}$ on $L^2(\mu)$, where we characterize the domain of the infinitesemal generator $L$ explicitly (Theorem~\ref{thm:semigroup}). Under mild conditions it is established that the resolvent of $(P(t))_{t \geq 0}$ is compact with immediate strong implications for the spectrum.
Then, in Section~\ref{sec:generatorspectrum} the spectrum of the zigzag semigroup is studied under the assumption of unimodality of $e^{-U}$ and vanishing refreshment, $\lambda_{\refr}(x) = 0$ for all $x$. It turns out that in this case we can find a holomorphic function $Z : \C\rightarrow \C$ such that the spectrum of $L$ is identical to the roots of $Z$ (Theorem~\ref{thm:spectrum}). Under a further suitable polynomial growth condition on the potential function, a spectral mapping theorem is shown to hold, so that the spectrum of $(P(t))_{t \geq 0}$ can expressed in terms of the roots of $Z$ (Theorem~\ref{thm:spectralmapping}), yielding in particular a characterization of the spectral gap of $(P(t))_{t \geq 0}$ in $L^2(\mu)$ (Theorem~\ref{thm:decay}). 

Interestingly, if $U$ is symmetric, as discussed in Section~\ref{sec:symmetric}, then the state space $L^2(\mu)$ can be separated into two subspaces that are invariant with respect to the zigzag semigroup. In particular this implies a decomposition of the spectrum (Theorem~\ref{thm:spectrum-symmetric}).
In Section~\ref{sec:perturbation} we determine the effect of small perturbations of the generator on the spectrum of the zigzag semigroup, allowing us to study the effect of a small non-zero refreshment intensity. 

In Section~\ref{sec:examples} some fundamental examples are discussed along with a numerical illustration of the spectrum. Interestingly, in the examples we consider (including a Gaussian target distribution) a slight increase above zero in the zigzag refreshment intensity results in a larger spectral gap and thus an improved convergence in $L^2$. In contrast, the asymptotic variance increases when the refreshment rate increases \cite{BierkensDuncan2016}. This apparent contradiction between convergence in law and convergence of empirical averages is also observed in reversible, nearly-periodic Markov chains \cite{Rosenthal2003} and non-reversible Markov chains \cite{Vialaret2020}.

Finally we include an appendix containing several technical lemmas.

\section{Notation and assumptions}
\label{sec:notation-assumptions}

Let $E = \R \times \{-1,+1\}$ and equip $E$ with the product topology generated by the usual topologies on $\R$ and $\{-1,+1\}$. We also equip $E$ with the product $\sigma$-algebra generated by products of Lebesgue sets in $\R$ and all subsets of $\{-1,+1\}$.

Associated to a \emph{potential function} $U : \R \rightarrow \R$,  which we assume to be continuously differentiable, we define the \emph{switching intensity} function
\begin{equation}\label{eq:switching-rate} \lambda(x,\theta) := (\theta U'(x) \vee 0) + \lambda_{\refr}(x), \quad (x,\theta) \in E.\end{equation}
Here $(a \vee b) := \max(a,b)$ for any $a, b \in \R$, and $\lambda_{\refr} : \R \rightarrow [0,\infty)$ is a continuous function. The function $\lambda_{\refr}$ is referred to as the \emph{refreshment rate} or \emph{excess switching rate}. 
%
In order to obtain the main results in this paper we will often assume that $\lambda_{\refr}(x) = 0$ for all $x \in \R$; this assumption will always be made explicit.

The switching intensities $\lambda$ define (under suitable conditions) a Markov process in $E$, with associated Markov semigroup $(P(t))_{t \geq 0}$ on $B(E)$, the Banach space of bounded measurable functions on $E$. We call $(P(t))$ the \emph{zigzag semigroup}.  For a probabilistic construction see, e.g., \cite{BierkensFearnheadRoberts2016,BierkensRoberts2017}.

It is the aim of this paper to carry out a detailed investigation of properties of this semigroup as a strongly continuous semigroup in $L^2(\mu)$, where $\mu$ denotes a stationary distribution of the process (to be made precise later). We will be particularly interested in spectral properties of the semigroup.

\subsection{Notation}

For a measure $\mu$ on a measurable space $E$ we write $\mu(f) = \int_E f \, d \mu$, for $f \in L^1(\mu)$. Let $E = \R \times \{-1,+1\}$. For $f : E \rightarrow \C$ we write $f^{\theta}  = f(\cdot, \theta)$, $\theta = \pm 1$.  
Recall the Sobolev spaces $W^{k,p}(\Omega, \nu)$ for $\Omega$ an open subset of $\R$ and $\nu$ a Borel measure on $\Omega$ and write $W^{k,p}(\Omega)$ when the underlying measure is Lebesgue measure.
Let $\mu$ denote a Borel measure on $E$. 
We say that $f \in W^{k,p}(\mu)$ if $f(\cdot, \theta) \in W^{k,p}(\R,\mu)$ for $\theta \in \{-1,+1\}$. Similarly we say that $f \in W^{k,p}_{\loc}(E)$ if, for any $\theta \in \{-1,+1\}$, $f(\cdot, \theta)$ restricted to $F$ is in $W^{k,p}(F)$ for any compact set $F \subset \R$.
Mappings $M$ of functions $f : \R\rightarrow \C$ yielding a function $M f : \R\rightarrow \C$ may be extended to mappings of functions $f : E \rightarrow \C$ by defining $(M f)^{\theta} = Mf (\cdot, \theta) = M f^{\theta}$ for $\theta = \pm 1$. For example, 
for $f \in W^{1,p}(E)$ we write $\partial_x f :  E \rightarrow \R$ for the function such that $(\partial_x f)^{\theta} = \partial_x (f^{\theta})$ for $\theta = \pm 1$.
Similarly a function space of functions $f : \R \rightarrow \R$ is readily extended to functions $f : E \rightarrow \R$. For example, $C_c^{\infty}(E)$ denotes the space of all functions $f : E \rightarrow \C$ such that $f^{\theta} \in C_c^{\infty}(\R)$ for $\theta = \pm 1$.

\subsection{Assumptions on the potential function}
\label{sec:assumptions}
We introduce several different assumptions on the potential function $U$ for later reference.
See Section~\ref{sec:examples} for the verification of these conditions for a specific family of potential functions.

%


\begin{assumption}{A1}
	\label{ass:sobolev}
	$U \in C^2(\R)$, $e^{-U} \in L^1(\R)$ and, for some $\delta \in (0,1)$, $m \in \R$, and $M > 0$,
	\[ m \leq  U''(x) \leq \delta |U'(x)|^2 + M, \quad \text{$x \in \R$.}\]
	Furthermore there are constants $c_1, c_2 \geq 0$ such that $\lambda_{\refr}(x) \leq c_1 + c_2 |U'(x)|$ for all $x \in \R$.
\end{assumption}

\begin{assumption}{A2}
	\label{ass:compactness}
	$\lim_{x \rightarrow \infty} U'(x) = + \infty$ and $\lim_{x \rightarrow -\infty} U'(x) = -\infty$.
\end{assumption}



\begin{assumption}{A3}[Unimodality]
	\label{ass:unimodal}
	$U(0) = 0$, $U'(x) \leq 0$ for $x \leq 0$ and $U'(x) \geq 0$ for $x \geq 0$.
\end{assumption}


\begin{assumption}{A4}[Growth condition]
	\label{ass:polynomial-tail}
	There are constants $C \leq 0$, $M \geq 0$, $m > 0$ and $p > 1$ such that 
	\begin{equation}
	\label{eq:polynomial-tail} U(y) \geq U(x) + C + m|x-y|^p \quad \text{for all $x, y$ for which  $M \leq x \leq y$ or $y \leq x \leq -M$.}
	\end{equation}
\end{assumption}

Note that for continuously differentiable $U$,  Assumption~\ref{ass:polynomial-tail} cannot hold for $C > 0$. Indeed, if~\eqref{eq:polynomial-tail} would hold with  $C > 0$ then for all $x \geq M$
\[ U'(x) = \lim_{h \downarrow 0} \frac{U(x+h) - U(x)}{h} \geq \lim_{h \downarrow 0}\frac C h + m h^{p-1} = +\infty. \]
\begin{remark}
If Assumption~\eqref{eq:polynomial-tail} holds for some $M > 0$, then it also holds for $M = 0$ (for a different value of $m$), which simplies proofs later on. This result is delegated to the appendix, see Lemma~\ref{lem:polynomial-tail}.
\end{remark}

Examples of potential functions for which Assumption~\ref{ass:polynomial-tail} is satisfied are given by $U(x) = |x|^p$ for $p > 1$, with $C = 0$ and $m = 1$. This can be seen by defining, for fixed $x \geq 0$, $f(y) := |y|^p - |x|^p - |y-x|^p$ and noting that $f'(y) \geq 0$ for $y \geq x$.

\begin{remark}
For convenience of the reader we provide the following lemma which gives an easily verifiable condition for Assumption~\ref{ass:polynomial-tail} to hold.
The reverse implication does not hold: see Lemma~\ref{lem:assumptions-family} for an example where the conditions of Lemma~\ref{lem:implications-m-convexity} are not satisfied whereas Assumptions~\ref{ass:compactness} and~\ref{ass:polynomial-tail} are satisfied.
\end{remark}



\begin{lemma}
	\label{lem:implications-m-convexity}
	Suppose $U \in C^2(\R)$ and $U''(x) \geq m$ for all $|x| \geq M$, where $m > 0$ and $M \geq 0$.
 Then Assumptions~\ref{ass:compactness} and~\ref{ass:polynomial-tail} hold.
\end{lemma}
\begin{proof}
The implication of Assumption~\ref{ass:compactness} is immediate.
	If the assumptions of the lemma are satisfied then $U'(x) \geq 0$ for $x$ sufficiently large (say for $x \geq \widetilde M \geq M$). By integrating it follows that for $y \geq x \geq \widetilde M$ and $\widetilde m = m/2$ we have that 
	\[ U(y) \geq U(x) + U'(x) (y-x) +  \widetilde m  (y-x)^2 \geq U(x) + \widetilde m  (y-x)^2.\]
	We may use an analogous argument for $y \leq x \leq -M$ to obtain Assumption~\ref{ass:polynomial-tail} for $C = 0$ and $p = 2$. 
\end{proof}

In Section~\ref{sec:symmetric} we will investigate the implications of the following symmetry assumption.

\begin{assumption}{A5}[Symmetry]
	\label{ass:symmetric}
	$U(x) = U(-x)$ for $x \in \R$.
\end{assumption}

\begin{remark}
 For convenience of the reader we remark that the following list of assumptions implies all conditions mentioned so far, but is somewhat stronger than necessary. In particular (iii) and (v) are stronger than necessary for most of the results.

 	\begin{itemize}
		\item[(i)] $U \in C^2(\R)$.
		\item[(ii)]  For some $\delta \in (0,1)$, and $M > 0$, 
		\[ U''(x) \leq \delta |U'(x)|^2 + M, \quad x \in \R.\]
		\item[(iii)] Symmetry: $U(x) = U(-x)$ for $x \in \R$.
		\item[(iv)] Unimodality: $U(0) = 0$ and $U'(x) \geq 0$ for $x \geq 0$.
		\item[(v)] $m$-strong convexity: $U''(x) \geq m$ for all $|x| \geq M$, where $m > 0$ and $M \geq 0$.
	\end{itemize}
In particular note that $e^{-U} \in L^1(\mathbb R)$ is implied by (v).
\end{remark}

\begin{example}[Gaussian distribution]
	Suppose $U(x) = \tfrac {x^2}{2 \sigma^2}$ with $\sigma> 0$. Then Assumptions~\ref{ass:sobolev}, \ref{ass:compactness}, \ref{ass:unimodal}, \ref{ass:polynomial-tail} and \ref{ass:symmetric} are satisfied. Here Assumption~\ref{ass:polynomial-tail} follows, e.g., by Lemma~\ref{lem:implications-m-convexity}.  We will study a more general family of distributions, including the Gaussian distribution, in Section~\ref{sec:examples}. 
\end{example}

\begin{example}[$t$-distribution]
	Suppose $U(x) = \left(\frac{\alpha+1}{2} \right) \log \left( 1 + \frac {x^2}{\alpha} \right)$, with $\alpha > 0$. In this case the probability measure on $\mathbb R$ with  density proportional to $e^{-U}$ is the $t$-distribution with $\alpha$ degrees of freedom. In particular if $\alpha = 1$ this is the Cauchy distribution.  Then $U'(x) = \frac{(\alpha+1)x}{\alpha + x^2}$ and $U''(x) = \frac{(\alpha + 1)(\alpha - x^2)}{(x^2 + \alpha)^2}$.   Since $U''$ is a bounded function for any $\alpha > 0$, Assumption~\ref{ass:sobolev} is satisfied. However, $|U'(x)| \rightarrow 0$ as $|x| \rightarrow \infty$ for any $\alpha > 0$, so that Assumption~\ref{ass:compactness} is not satisfied. Assumptions~\ref{ass:unimodal} and \ref{ass:symmetric} are trivially satisfied; Assumption~\ref{ass:polynomial-tail} is not satisfied.
\end{example}

\section{The zigzag semigroup in $L^2(\mu)$}
\label{sec:semigroup}

In this section we take a purely functional analytic approach to the construction of the zigzag semigroup in $L^2(\mu)$. That is, we do not refer to the probabilistic construction of the process as carried out, e.g., in \cite{BierkensRoberts2017}. This construction is discussed in Section~\ref{sec:construction}. Next, we establish that under reasonable conditions (i.e. Assumptions~\ref{ass:sobolev}, \ref{ass:compactness}) the semigroup has a compact resolvent, which is important in the characterization of the spectrum from Section~\ref{sec:generatorspectrum} onwards.

\subsection{Construction of the zigzag semigroup}
\label{sec:construction}

Define measures $\nu$ on $\R$ and $\mu$ on $E$, respectively by 
\[ \nu(A) = \int_A e^{-U(x)} \, d x, \quad \text{and} \quad  \mu(A \times \{ \theta\}) = \nu(A), \quad A \in \mathcal B(\R), \, \theta \in \{-1,+1\}.\]
At this point there is no requirement that $e^{-U}$ be integrable. Therefore the measure $\mu$ is not necessarily finite, but as a consequence of the continuity of $U$ it is $\sigma$-finite. 
We consider the complex Hilbert space $L^2(\mu)$ with inner product $\langle \cdot, \cdot \rangle$ and norm $\|\cdot\|$. The space $L^2(\mu)$ is separable by \cite[p. 36]{Davies2007}.
Define a linear operator $(L, \mathcal D(L))$ on the domain
\begin{equation}
\label{eq:domain}
\mathcal D(L) = \{ f \in L^2(\mu) \cap W^{1,2}_{\loc}(E) : (x,\theta) \mapsto \theta \partial_x f(x,\theta) + \lambda(x,\theta) (f(x,-\theta) - f(x,\theta)) \in L^2(\mu) \}.
\end{equation}
by
\begin{equation}
\label{eq:generator}
L f(x,\theta) =  \theta \partial_x f(x,\theta) + \lambda(x,\theta) (f(x,-\theta) - f(x,\theta)).
\end{equation}
We introduce the \emph{flipping operator} $F f(x,\theta) = f(x,-\theta)$ and we will often use the shorthand notation $L f = \theta \partial_x f + \lambda (F f - f)$.

The aim of this section is to establish that $L$ is the generator of a  semigroup $(P(t))_{t \geq 0}$ in $L^2(\mu)$, which we will refer to as the \emph{zigzag semigroup}. In order to do so, we will first show that $L$ is a closed dissipative operator (Lemmas~\ref{lem:closed},~\ref{lem:dissipative}), determine a core for $L$ (Lemma~\ref{lem:core}), and next 
characterize the adjoint of $L$, finding essentially that $L^{\star} = FLF$ (Lemmas~\ref{lem:adjoint},~\ref{lem:adjoint-symmetry}).
These results then immediately imply the wellposedness of the Cauchy problem, i.e., the construction of the semigroup $P(t)$, see Theorem~\ref{thm:semigroup}.

\begin{lemma}
	\label{lem:closed}
	The operator $(L, \mathcal D(L))$ is closed.
\end{lemma}
\begin{proof}
	Let $(f_n) \subset \mathcal D(L)$ be a converging sequence in $L^2(\mu)$ with $\lim_{n \rightarrow \infty} f_n = f \in L^2(\mu)$, and suppose $(g_n) := (L f_n)$ converges in $L^2(\mu)$ to $g$.
	Note that \begin{equation} \label{eq:dfnx} \partial_x f_n(x,\theta) = \theta \Big(g_n(x,\theta) - \lambda(x,\theta) (f_n(x,-\theta) - f_n(x,\theta))\Big)\end{equation} and define in similar spirit 
	\[ h(x,\theta) := \theta \Big( g(x,\theta) - \lambda(x,\theta) (f(x,-\theta) - f(x,\theta)) \Big).\] We will first show that $f \in W^{1,2}_{\loc}(E)$ and that $\partial_x f = h$. To this end let $\phi \in C_c^{\infty}(E)$. Then
	\begin{align*} -\int_E f \partial_x \phi \, d x & = -\int_E \left(\lim_{n \rightarrow \infty} f_n \right) \partial_x \phi \, d x = -\lim_{n \rightarrow \infty} \int_E f_n \partial_x \phi \, d x \\
	& =\lim_{n \rightarrow \infty} \int_E \left(\partial_x f_n\right) \phi \, d x = \int_E \left(\lim_{n \rightarrow \infty} \partial_x f_n \right) \phi \, d x = \int_E h \phi \, d x. \end{align*}
	In he above derivation we were allowed to exchange the limit and integral twice, since (for the first interchange), $f_n \rightarrow f$ in $L^2(\mu)$, and using (for the second interchange) that $f_n \rightarrow f$ in $L^2(\mu)$, $g_n \rightarrow g$ in $L^2(\mu)$, and $\lambda$ is bounded on compact intervals, and thus $\partial_x f_n$ converges in $L^2_{\text{loc}}(\R)$ by~\eqref{eq:dfnx}.
	This establishes that $\partial_x f = h$. Furthermore, taking a subsequence, $\theta \partial_x f_{n_k} =  g_{n_k} - \lambda(F f_{n_k} - f_{n_k})\rightarrow g - \lambda(F f - f) = \theta \partial_x f$ almost everywhere (by $L^2$-convergence of $f_n$ and $g_n$). Therefore $g_{n_k} = L f_{n_k} \rightarrow \theta \partial_x f + \lambda (F f - f) = g$ almost everywhere. Since also $g_{n_k} \rightarrow g$ in $L^2(\mu)$ we have $f \in \mathcal D(L)$ and $Lf=g$.
\end{proof}

\begin{lemma}
 \label{lem:core}
$C_c^{\infty}(E)$ is a core for $L$.
\end{lemma}

\begin{proof}
We use ideas from the proof of \cite[Proposition 8.7.3]{LorenziBertoldi2007}. Let $f \in \mathcal D(L)$, and let $\xi : [0,\infty) \rightarrow [0,1]$ be a smooth function such that $\xi = 1$ in $[0,1/2]$ and $\xi = 0$ in $[1, \infty)$, and define $\xi_n(x) = \xi(|x|/n)$. Then define $f_n(x,\theta) = f(x,\theta) \xi_n(x)$. Then every $f_n$ has compact support, $f_n \in W^{1,2}(\mu)$, and $f_n \rightarrow f$ in $L^2(\mu)$ by dominated convergence.
Furthermore we have
\begin{align*} & L f_n(x,\theta) 
 = \theta \xi_n'(x)  f(x,\theta)  + \xi_n(x) L f(x,\theta).
\end{align*}
Note that $\sup_{x \in \R} |\xi_n'(x)| \leq C /n$ for some $C > 0$.
It follows that, as $n \rightarrow \infty$,
\[ \| L f_n(x,\theta) - L f(x,\theta) \|_{L^2(\mu)} \leq C/n \|f\|_{L^2(\mu)} + \| (1 - \xi_n) L f\|_{L^2(\mu)} \rightarrow 0, \]
where the second term converges to zero by dominated converence. Thus, functions with compact support are dense in $\mathcal D(L)$ with respect to the graph norm of $L$. We can now limit ourselves to the situation in which $f$ has compact support. Let $\varphi \in C_c^{\infty}(\R)$ be a smooth function, compactly supported in $[-1,1]$ such that $\varphi = 1$ in $[-1/2,1/2]$, $\varphi(x) \in [0,1]$ for all $x$, and $\| \varphi\|_{L^1(\R)} = 1$. Set $\varphi_n(x) = n \varphi(x/n)$ for $n \in \N$, and for compactly supported $f$, define
\[ f_n(x,\theta) = (f \star \varphi_n)(x,\theta) := \int_{\R} \varphi_n(x -y) f(y,\theta) \, d y, \quad x \in \R, \quad n \in \N.\]
Then $f_n \in C_c^{\infty}(E)$, uniformly supported in $\operatorname{supp}(f) + B(1)$, and $f_n \rightarrow f$ in $L^2(\mu)$. Furthermore, since $\partial_x f_n = (\partial_x f) \star \varphi_n$, and $\lambda$ is bounded uniformly in $n$ on the support of $(f_n)$, it follows that $\| L f_n - L f \|_{L^2(\mu)} \rightarrow 0$.
\end{proof}

\begin{lemma}[Dissipativity]
	\label{lem:dissipative}
	$(\lambda^+ + \lambda^-)|f^+ - f^-|^2 \in L^1(\nu)$ for any $f \in \mathcal D(L)$,  and
	\begin{equation} \label{eq:dirichlet-form} \Re \int_{E} (L f) \overline f \, d \mu \leq - \tfrac 1 2 \int_{\R} (\lambda^+ + \lambda^-)|f^+ - f^-|^2 \, d \nu, \quad f \in  \mathcal D(L).\end{equation}
\end{lemma}
\begin{proof}
First we establish the result for $f \in C_c^{\infty}(E)$. 
	By partial integration, using compact support of $f$, and noting that~\eqref{eq:switching-rate} leads to the equality
	$\lambda^+(x) - \lambda^-(x) = U'(x)$, we find that 
	\begin{align*}
	& \Re \sum_{\theta = \pm 1} \int_{-\infty}^\infty (L f)^{\theta} \overline{f^{\theta}} \, d \nu(x) \\
	& = \Re  \int_{-\infty}^\infty \left\{\left(\partial_x f^+ + \lambda^+ (f^- - f^+) \right) \overline{f^+} + \left(-\partial_x f^- + \lambda^- (f^+ - f^-) \right) \overline{f^-} \right\}  e^{-U} \, d x \\
	& = \int_{-\infty}^\infty \tfrac 1 2 \partial_x (|f^+|^2 e^{-U}) + \Re \left(\tfrac 1 2 U' f^+  + \lambda^+ (f^- - f^+) \right) \overline{f^+} e^{-U} \, dx \\
	& \quad  + \int_{-\infty}^\infty (-\tfrac 1 2 \partial_x (|f^-|^2 e^{-U})  + \Re \left(-\tfrac 1 2  U' f^- + \lambda^- (f^+ - f^-) \right) \overline{f^-}  e^{-U} \, d x \\
	& = \Re  \int_{-\infty}^\infty \left( \lambda^+ f^- \overline{f^+} + \lambda^- f^+ \overline{f^-} - \tfrac 1 2 (\lambda^+ + \lambda^-) |f^+|^2 - \tfrac 1 2 (\lambda^+ + \lambda^-)|f^-|^2  \right) e^{-U} \,dx \\
	& =  - \tfrac 1 2 \int_{-\infty}^\infty (\lambda^+ + \lambda^-) |f^+ - f^-|^2 e^{-U} \, d x.
	\end{align*}
	
	For $f \in \mathcal D(L)$, using that $C_c^{\infty}(E)$ is a core for $L$ (Lemma~\ref{lem:core}),  take a sequence $f_n \rightarrow f$, with $f_n \in C_c^{\infty}(E)$, and with convergence in the graph norm of $L$.
	By Fatou, 
	\[ \tfrac 1 2\int_{\R} (\lambda^+ +\lambda^-) |f^+ - f^-|^2 \, d \nu \leq - \Re \int_E (L f) \overline f \, d \mu < \infty.\]
\end{proof}

Recall that the adjoint operator of a densely defined operator $(A, \mathcal D(A))$ on $L^2(\mu)$ is defined on the domain
\[ \mathcal D(A^{\star}) = \{ g \in L^2(\mu) : \text{there is an $h \in L^2(\mu)$ such that  $\langle A f, g \rangle = \langle f, h\rangle$ for all $f \in \mathcal D(A)$}\}, \]
by $A^{\star} g = h$, with $h$ the function as described in the domain of the adjoint.

Define an operator $(L^{\times}, \mathcal D(L^{{\times}}))$ by
\[ \mathcal D(L^{{\times}}) = \{ g \in L^2(\mu) \cap W^{1,2}_{\loc}(E): (x,\theta) \mapsto - \theta \partial_x g(x,\theta) + \lambda(x,-\theta) (g(x,-\theta) - g(x,\theta)) \in L^2(\mu) \}\]
and
\[ L^{{\times}} g (x,\theta) =  - \theta \partial_x g(x,\theta) + \lambda(x,-\theta) (g(x,-\theta) - g(x,\theta)), \quad g \in \mathcal D(L^{{\times}}), \, (x,\theta) \in E.\]
This operator is going to be our candidate for the adjoint operator $L^{\star}$.
\begin{lemma}
	\label{lem:adjoint}
	$(L^{\star},\mathcal D(L^{\star}))= (L^{\times},\mathcal D(L^{{\times}}))$.
\end{lemma}
\begin{proof}
    We first show that $\mathcal D(L^{\star}) \subset \mathcal D(L^{\times})$, and that $L^{\star}$ conincides with $L^{\times}$ on its domain.
	Suppose $g \in \mathcal D(L^{\star})$, and let $h \in L^2(\mu)$ be so that $\langle Lf, g \rangle = \langle f, h\rangle$ for every $f \in \mathcal D(L)$. Note that $C_c^{\infty}(E) \subset \mathcal D(L)$. For $f \in C_c^{\infty}(E)$ and $\theta = \pm 1$, we have 
	\begin{align*} \sum_{\theta = \pm 1} \int_{\R} \theta \partial_x f(x,\theta) \overline{ g(x,\theta)} e^{-U(x)} \, d x & = \langle \theta \partial_x f,  g \rangle \\
	& = \langle (L f - \lambda(F f-f)), g \rangle = \langle  f, h \rangle - \langle \lambda (F f - f ), g \rangle\\
	&  = \sum_{\theta = \pm 1} \int_{\R} f \overline{\left(h - F(\lambda g) + \lambda g\right)}e^{-U} \, d x.\end{align*}
	Since this holds in particular for any $f$ such that either $f(\cdot, + 1) =0$ or $f(\cdot, -1) =0$, it follows that $\theta g e^{-U}$ has a weak derivative $ - \left( h - F(\lambda g) + \lambda g \right)e^{-U}$. From this we may conclude (by Lemma~\ref{lem:productrule}) that $g = (\theta g e^{-U}) (\theta e^U) \in W^{1,2}_{\loc}(E)$ with weak derivative
	\[ \partial_x g = - \theta \left( h - F(\lambda g) + \lambda g \right) + U' g.\]
	Isolating $h$ on the left hand side gives
	\begin{align*} h(x,\theta) & = - \theta \partial_x g(x,\theta) + \lambda(x,-\theta) g(x,-\theta) - \lambda(x,\theta) g(x,\theta) + \theta U'(x) g(x,\theta) \\
	&  = - \theta \partial_x g(x,\theta) + \lambda(x,-\theta) (g(x,-\theta) -  g(x,\theta)),
	\end{align*}
	using that $\theta U' = \lambda - F \lambda$. 
	Since $h \in L^2(\mu)$ the right-hand side is also in $L^2(\mu)$. So $g \in \mathcal D(L^{{\times}})$ and $h = L^{{\times}} g$. 
	
	Conversely, if $g \in \mathcal D(L^{\times})$ and $f \in \mathcal D(L)$, we may approximate $f$ by $\tilde f \in C^{\infty}_c(E)$. By partial integration it follows that $\langle L \tilde f, g\rangle = \langle \tilde f, L^{\times} g \rangle$. Since $\tilde f$ can be arbitrarily close to $f$ in the graph norm of $L$ by Lemma~\ref{lem:core}, it follows that $\langle L f, g \rangle = \langle f, L^{\times} g \rangle$. This establishes that $\mathcal D(L^{\times}) \subset \mathcal D(L^{\star})$.
\end{proof}

The following lemma establishes that $L^{{\times}}  = FLF$. 

\begin{lemma} \label{lem:adjoint-symmetry}
	$f \in \mathcal D(L)$ if and only if $F f \in \mathcal D(L^{{\times}})$, and $FL f = L^{{\times}} F f$ for $f \in \mathcal D(L)$. In particular $L^{\star} = L^{\times}$ is dissipative.
\end{lemma}
\begin{proof}
	If $f \in \mathcal D(L)$, then $f \in W^{1,2}_{\loc}(E)$, and therefore $F f \in W^{1,2}_{\loc}(E)$. Applying $L^{{\times}}$ (formally) to $Ff$ yields
	\[ L^{{\times}} F f(x,\theta) = -\theta \partial_x f(x,-\theta) + \lambda(x,\theta) (f(x,\theta) - f(x,-\theta)) = F L f(x,\theta).\]
	Since $L f \in L^2(\mu)$, it follows that $F L f \in L^2(\mu)$. This establishes that $F f \in \mathcal D(L^{{\times}})$ and $FL f = L^{{\times}} F f$ for $f \in \mathcal D(L)$.
	The reverse inclusion is analogous.
	Since $L$ is dissipative by Lemma~\ref{lem:dissipative}, this applies to $L^{\times}$.
\end{proof}

We summarize our findings in the following theorem.

\begin{theorem}
	\label{thm:semigroup}
	$L$ is a closed dissipative operator which generates a strongly continuous 
	contraction semigroup $(P(t))_{t \geq 0}$ in $L^2(\mu)$. The adjoint of $L$ is given by $L^{\star} = F L F$, i.e., $\mathcal D(L^{\star}) = \{ F f : f \in \mathcal D(L)\}$ and $L^{\star} g = FLF g$ for $g \in \mathcal D(L^{\star})$. The adjoint semigroup $(P^{\star}(t))_{t \geq 0}$ is given by $P^{\star}(t) = F P(t) F$ for $t \geq 0$.
\end{theorem}

\begin{proof}
	The fact that $L$ generates a strongly continuous semigroup  follows by \cite[Theorem II.3.17]{EngelNagel2000}, using that $L$ and $L^{\star}$ are dissipative. In particular the range of $\gamma - L$ is equal to $L^2(\mu)$ (\cite[Theorem II.3.15]{EngelNagel2000}) for any $\gamma > 0$. 
The statement on the adjoints of $L^{\star}$ and $P^{\star}$ follows by Lemma~\ref{lem:adjoint}.
\end{proof}

\subsection{Compactness of the resolvent of the zigzag semigroup generator}
\label{sec:compactness}
We will now investigate conditions under which the resolvent of the generator of the zigzag semigroup is compact. This is of fundamental importance in the characterization of the spectrum in subsequent sections.

Let $W^{1,2}(\nu)$ denote the usual Sobolev space of functions in $L^2(\nu)$ which have a weak derivative in $L^2(\nu)$. Since $W^{1,2}(\nu) \subset W^{1,2}_{\text{loc}}(\R)$, all $f \in W^{1,2}(\nu)$ are absolutely continuous.  Finally we denote by $W^{1,2}(\mu)$ the (equivalence classes of) measurable functions $f : E \rightarrow \R$ such that $f^{\theta} \in W^{1,2}(\nu)$ for $\theta  = \pm 1$. 

We may obtain a better understanding of $\mathcal D(L)$ if we suppose Assumption~\ref{ass:sobolev} and~\ref{ass:compactness} are satisfied. In particular, $e^{-U}$ is assumed to belong to $L^1(\R)$ so that $\nu$ and $\mu$ are finite measures. The relevance of these assumptions follows for a large part from the following lemmas.

\begin{lemma}[{\cite[Lemma 8.5.2, Theorem 8.5.3]{LorenziBertoldi2007}}]
	\label{lem:sobolev-classical}
	Suppose that Assumption~\ref{ass:sobolev} holds. Then for any $p \in [2, +\infty)$ there exists a positive constant $C$ such that
	\[ \int_{\R} |f|^p |U'|^2 d \nu \leq C \|f\|_{W^{1,p}(\nu)}^p, \quad f \in W^{1,p}(\nu).\]
	If additionally Assumption~\ref{ass:compactness} is satisfied, then the embedding $W^{1,p}(\nu) \subset L^p(\nu)$ is compact for any $p \in [2,+\infty)$.
\end{lemma}

\begin{lemma}
\label{lem:sobolev-norm}
Suppose that $U$ satisfies Assumption~\ref{ass:sobolev}.
Suppose $f \in W^{1,2}_{\loc}(\R) \cap L^2(\nu)$ is such that $f' - U' f \in L^2(\nu)$. Then $f \in W^{1,2}(\nu)$ and 
\begin{equation}
 \label{eq:sobolev-norm} 
 \int_{\R} |f'|^2 \, d \nu \leq -m \int_{\R} |f|^2 \, d \nu + \int_{\R} |f' - U' f|^2 \, d \nu < \infty.
\end{equation}
\end{lemma}
\begin{proof}
First suppose $f \in C_c^{\infty}(\R)$.
We compute
\begin{align*}
\int_{\R} |f' - U' f |^2 \, d \nu = \int_{\R} \left\{ |f'|^2 - f' U' \overline f - U' f (\overline{f'}) +|U'|^2 |f|^2 \right\} \, d \nu.
\end{align*}
For the cross terms we find by partial integration that
\begin{align*}
    \int_{\R} \{  - (f') U' \overline f - U' f (\overline{f'}) \} \, d \nu & = - \int_{\R}  (\partial_x |f|^2) U' \, e^{-U} \, d x \\
    & = \int_{\R} \left\{ |f|^2 U'' - |U'|^2 |f|^2 \right\} e^{-U} \, d x.
\end{align*}
Thus,
\begin{align*}
\int_{\R} |f'-U' f |^2 \, d \nu = \int_{\R} \left\{ |f'|^2 + |f|^2 U'' \right\} d \nu. 
\end{align*}
Combining with the assumption that $U'' \geq m$ and rearranging yields the stated conclusion for $f \in C_c^{\infty}(\R)$.

Analogous to the proof of~Lemma~\ref{lem:core}, it may be established that $C_c^{\infty}(\R)$ is a core for the closed operator $(C, \mathcal D(C))$ on $L^2(\nu)$ with $C f = f' - U' f$ and $\mathcal D(C) = \{ f \in L^2(\nu) \cap W^{1,2}_{\loc} : C f \in L^2(\nu)\}$; see Lemma~\ref{lem:C-core} for details. Letting $f_n \rightarrow f$ in the graph norm of $C$ with $(f_n) \subset C_c^{\infty}(\R)$, we find that, using Fatou's lemma, and after taking a subsequence for which $f'_n$ converges almost everywhere, 
\begin{align*}
\| f'\|^2_{L^2(\nu)} & = \| \lim_{n \rightarrow \infty} f_n'\|^2_{L^2(\nu)} \\
& \leq 
 \liminf_{n \rightarrow \infty} \| f'_n\|^2_{L^2(\nu)} \\
 & \leq \liminf_{n \rightarrow \infty} \left( -m \| f_n\|_{L^2(\nu)}^2 + \| C f_n \|^2_{L^2(\nu)}\right) \\
 & = -m \| f\|_{L^2(\nu)}^2 + \| C f \|^2_{L^2(\nu)}  <\infty.
\end{align*}
 \end{proof}

We can now state sufficient conditions for the resolvent of the generator of the zigzag semigroup to be compact.

\begin{theorem}[Compact resolvent]
	\label{thm:compact-resolvent}
	Suppose Assumption~\ref{ass:sobolev} is satisfied. Then $\mathcal D(L) = W^{1,2}(\mu)$. If also  Assumption~\ref{ass:compactness} holds then the embedding $W^{1,2}(\mu) \subset L^2(\mu)$ is compact and in particular, for any $\gamma > 0$, the resolvent operator $(\gamma - L)^{-1}: L^2(\mu) \rightarrow L^2(\mu)$ is compact.
\end{theorem}
\begin{proof}
First we prove that $\mathcal D(L) \subset W^{1,2}(\mu)$.
Suppose $f \in \mathcal D(L)$ and write $ L f = h$. Write $g_1 = f^+ - f^-$ and $g_2 = f^+ + f^-$.
We have
\begin{align*}
 \partial_x g_1 & = h^+ - \lambda^+(f^- - f^+) + h^- -\lambda^-(f^+- f^-) = h^+ + h^- + U' g_1,\\
 \partial_x g_2 & = h^+  -\lambda^+(f^- - f^+) - h^- + \lambda^- (f^+ - f^-) = h^+ - h^- - (2 \lambda_{\refr} + |U'|) g_1.
\end{align*}
We see that $g_1$ satisfies $g_1' - U' g_1 = h^+ + h^- \in L^2(\nu)$ so that, by Lemma~\ref{lem:sobolev-norm}, $g_1 \in W^{1,2}(\nu)$.
Next by the estimate on $\lambda_{\refr}$ of Assumption~\ref{ass:sobolev}, Lemma~\ref{lem:sobolev-classical} and the second equation, $g_2 \in W^{1,2}(\nu)$. The inclusion $\mathcal D(L) \subset W^{1,2}(\mu)$ follows since $f^{\pm}$ may be expressed as linear combinations of $g_1$ and $g_2$.

Conversely, we now prove that $W^{1,2}(\mu) \subset \mathcal D(L)$.
Since $\lambda_{\refr}(x) \leq c_1 + c_2 |U'(x)|$ for some constants $c_1, c_2$, it follows that $|\lambda(x,\theta)| \leq C_1 + C_2 |U'(x)|$ for some constants $C_1$, $C_2$. 
By Assumption~\ref{ass:sobolev} and Lemma~\ref{lem:sobolev-classical}, $\| f \|_{L^2(\mu)} + \| L f\|_{L^2(\mu)} \leq \widetilde C \|f\|_{W^{1,2}(\mu)}$ for some $\widetilde C > 0$, which establishes that $W^{1,2}(\mu) \subset \mathcal D(L)$.

Let $\gamma > 0$. Recall that $\gamma - L : \mathcal D(L) \rightarrow L^2(\mu)$ is bijective, so the inverse map exists, and is bounded by dissipativity (see \cite[Proposition II.3.14]{EngelNagel2000}. 
Since the embedding $W^{1,2} \hookrightarrow L^2(\mu)$ is compact by Lemma~\ref{lem:sobolev-classical} it follows from \cite[Proposition II.4.25]{EngelNagel2000} that the resolvent is compact.
\end{proof}

\begin{remark}
If Assumption~\ref{ass:sobolev} is satisfied, then~\eqref{eq:dirichlet-form} holds with equality, i.e., for all $f \in \mathcal D(L)$ we have
\[ \Re \int_{E} (L f) \overline f \, d \mu = - \tfrac 1 2 \int_{\R} (\lambda^+ + \lambda^-)|f^+ - f^-|^2 \, d \nu, \quad f \in  \mathcal D(L).\]
This follows from the proof of Lemma~\ref{lem:dissipative}. Indeed in the proof the above relation is established for $f \in C_c^{\infty}(E)$. Now for $f \in \mathcal D(L) = W^{1,2}(\mu)$, take a sequence $(f_n) \subset C_c^{\infty}(E)$ converging in $W^{1,2}(\mu)$ towards $f$, and note that both sides of the above equation converge, using Lemma~\ref{lem:sobolev-classical}.
\end{remark}

The fact that the resolvent is compact under suitable conditions leads naturally to the question whether the semigroup is eventually compact. We will now establish that this is not the case.  The following lemma first states in a technical way that the zigzag semigroup consists simply of translations in regions of the space where the switching intensity vanishes, a necessary ingredient for the conclusion of Proposition~\ref{prop:not-eventually-compact} that the zigzag semigroup is in general not eventually compact.

\begin{lemma}
	\label{lem:noswitching-is-just-translation}
	Suppose $\lambda_{\refr}(x) = 0$ for all $x$.
	Suppose $\theta U'(x) \leq 0$  for some $\theta \in \{-1,+1\}$ and all $x \in I$, where $I = [a,b]$ is a non-empty interval. Then $P(t) f(x,\theta) = f(x + \theta t,\theta)$ for all $t \geq 0$ and $\mu$-almost all $x \in \R$ for which $\{ x + \theta s : 0 \leq s \leq t\} \subset I$.
\end{lemma}
\begin{proof}
	First suppose $f \in \mathcal D(L)$. Write $\varphi(t,x) := P(t) f(x,\theta)$ with $\theta$ as in the assumption of the lemma. Then, for $x \in I$
	\[ \partial_t \varphi(t,x) = \partial_t P(t) f (x,\theta) = L P(t) f(x,\theta) = \theta \partial_x P(t) f(x,\theta) = \theta \partial_x \varphi(t,x).\]
	Fix $(x, t)$ such that $x + \theta s \in [a,b]$ for $0 \leq s \leq t$.
	Consider the characteristic curve $X(s) =x + \theta s$, $T(s) = t - s$.
	We find that, since $X(s) \in I$ for all $0 \leq s \leq t$, then
	$\frac{d}{ds} \varphi(T(s), X(s)) = 0$, so that 
	\[ \varphi(t, x) = \varphi(T(0),X(0)) =  \varphi(T(t), X(t)) = \varphi(0, x + \theta t) = f(x + \theta t).\]
	Now if $f_n \rightarrow f$ in $L^2(\mu)$ with  $(f_n) \subset \mathcal D(L)$, then $P(t) f_{n_k}(x,\theta) \rightarrow P(t) f(x,\theta)$, $\mu$-almost all $x$, for a subsequence $f_{n_k}$, and $f_{n_{k_l}}(x + \theta t, \theta) \rightarrow f(x + \theta t, \theta)$, $\mu$-almost all $x$, for a further subsequence, which yields the result for general $f$.
\end{proof}

\begin{proposition}
	\label{prop:not-eventually-compact}
	Suppose Assumption~\ref{ass:unimodal} is satisfied. Furthermore assume that $\lambda_{\refr}(x) = 0$ everywhere.
Then, for any $t > 0$, the operator $P(t)$ is not compact.
\end{proposition}

\begin{proof}
	For $n \in \N$  define functions $f_n$ by
	\begin{align*} \widetilde f_n^+(x) &= \1_{[-1/n,-1/(n+1)]}(x),  \quad & \widetilde f_n^-(x) &= 0, \quad &\widetilde f_n(\cdot,\pm 1)& = \widetilde  f_n^{\pm}, \quad &f_n &= \widetilde f_n / \|\widetilde f_n\|_{L^2(\mu)}.
	\end{align*}
	Then $(f_n)$ is a bounded sequence of functions in $L^2(\mu)$. Therefore if $P(t_0)$ were compact for some $t_0 > 0$, it should be possible to find a convergent subsequence of $(P(t_0)f_n)_{n=1}^{\infty}$. However we will show that this is impossible.
	
	By Assumption~\ref{ass:unimodal}, $U'(x) \leq 0$ for $x \leq 0$. By Lemma~\ref{lem:noswitching-is-just-translation},
	\[ P(t_0) f (x,+1) = f(x+t_0,+1) \quad \text{if $x \leq -t_0$, \quad for any $f \in L^2(\mu)$.}\]
	Suppose $m \neq n$. Then 
	\begin{align*}
	\| P(t_0) f_n - P(t_0) f_m\|_{L^2(\mu)}^2 & \geq  \int_{-\infty}^{-t_0} |P(t_0) f_n(x,+1) - P(t_0) f_m(x,+1)|^2 e^{-U(x)} \, d x  \\
	& = \int_{-\infty}^{-t_0} |f_n^+(x+t_0) - f_m^+(x+t_0)|^2 e^{-U(x)} \, d x \\
	& = \int_{-\infty}^{0} \left( |f_n^+(y)|^2 + |f_m^+(y)|^2 \right) e^{-U(y-t_0)} \, d y \\
	& = \frac{\int_{-1/n}^{-1/{(n+1)}} e^{-U(y-t_0)} \, d y}{\int_{-1/n}^{-1/{(n+1)}} e^{-U(y)} \, d y} + \frac{\int_{-1/m}^{-1/{(m+1)}} e^{-U(y-t_0)} \, d y}{\int_{-1/m}^{-1/{(m+1)}} e^{-U(y)} \, d y} \\
	&\geq 2 e^{-U(-1-t_0)+U(0)}.
	\end{align*}
	This establishes that no subsequence of $(P(t_0)f_n)_{n \in \N}$ is Cauchy, and therefore that $P(t_0)$ is not compact.
\end{proof}



\section{Spectral theory of the zigzag generator}
\label{sec:generatorspectrum}
This section will describe key results concerning the characterization of the spectrum of the zigzag semigroup under the assumptions of unimodularity and no refreshments. The main result is Theorem~\ref{thm:spectrum} which, in summary, gives a description of the spectrum of $L$ as consisting of the roots of a particular complex valued function.

As usual define the \emph{resolvent set} $\rho(A)$ of a closed operator $A$ by
\[ \rho(A) = \{ \gamma \in \C: \text{$\gamma -A$ is invertible and $(\gamma - A)^{-1}$ is bounded}\}\]
and the \emph{spectrum} $\sigma(A)$ by $\sigma(A) = \C \setminus \rho(A)$. The \emph{point spectrum} of $A$ is defined as
\[\sigma_p(A) = \{ \gamma \in \C: \text{$\gamma - A$ is not injective} \}.\]
Elements $\gamma \in \sigma_p(A)$ are called \emph{eigenvalues}.
For a disjoint decomposition of the spectrum $\sigma(A) = \sigma_c \cup \sigma_u$ where $\sigma_u$ is closed and  $\sigma_c$ is compact, the \emph{spectral projection} $P_c$ is defined as
\[ P_c = \frac 1 {2 \pi \i} \int_{\Gamma} (\zeta - A)^{-1} \, d \zeta,\]
where $\Gamma$ is a Jordan curve in the complement of $\sigma_u$ and enclosing $\sigma_c$. See \cite[Section IV.1]{EngelNagel2000} or \cite[Section III.6.4]{Kato1995} for details. If $\Gamma$ encloses only a single element $\gamma \in \sigma(A)$ then we write $P_{\gamma} = P_c$.

We start with a simple corollary of Theorem~\ref{thm:compact-resolvent}.
\begin{corollary}
	\label{cor:compact-resolvent}
	Suppose Assumptions~\ref{ass:sobolev} and~\ref{ass:compactness} are satisfied. Then the spectrum of $L$ consists entirely of isolated eigenvalues with finite multiplicities, and the resolvent operator $(\gamma - L)^{-1}$ is compact for every $\gamma \in \rho(L)$.
\end{corollary}
\begin{proof}
	This follows from Theorem~\ref{thm:compact-resolvent} and \cite[Theorem III.6.29]{Kato1995}.
\end{proof}

There are also some interesting immediate consequences of the fact that $L^{\star} = F L F$ which we state in some generality in Lemma~\ref{lem:J-selfadjoint}.
Let $(\mathcal H, \langle \cdot, \cdot, \rangle)$ denote a Hilbert space, and let $J$ be a unitary operator on $\mathcal H$.
We call a closed operator $(A, \mathcal D(A))$ on $\mathcal H$ \emph{$J$-selfadjoint} if $A^{\star} = J A J^{\star}$. 
In the case of the zigzag process we have that $L$ is $J$-selfadjoint for $J = F$.

\begin{lemma}
	\label{lem:J-selfadjoint}
	Suppose $(A, \mathcal D(A))$ is a closed operator on $\mathcal H$ such that $A$ is $J$-selfadjoint. 
	\begin{itemize}
		\item[(i)] $\rho(A)=\rho(A^{\star})$.
		\item[(ii)] $\sigma(A) = \{ \overline \gamma : \gamma \in \sigma(A)\}$ and $\sigma_p(A) = \{ \overline \gamma : \gamma \in \sigma_p(A)\}$ (i.e. $\sigma(A)$ and $\sigma_p(A)$ are closed under taking complex conjugate).
		\item[(iii)]  If $\gamma_1$, $\gamma_2$, $\gamma_1 \neq \gamma_2$, are eigenvalues of $A$ with eigenvectors $\phi_1$, $\phi_2$, respectively, then either $\gamma_1 = \overline \gamma_2$ or
		$\phi_1$ and $\phi_2$ are $J$-orthogonal, i.e., $\langle J \phi_1, \phi_2 \rangle = 0$.
		%
		\item[(iv)]
		If $\phi$ is an eigenvector of $A$ with eigenvalue $\gamma$, then $J \phi$ is an eigenvector of $A^{\star}$ with eigenvalue $\gamma$.
		\item[(v)] The adjoint of the spectral projection $P_{\gamma}$ associated to an isolated eigenvalue $\gamma \in \sigma(A)$ satisfies
		\[ (P_{\gamma})^{\star} = J P_{\overline\gamma} J^{\star}\]
		and is equal to the spectral projection associated with $A^{\star}$ for the eigenvalue $\overline \gamma$.
	\end{itemize}
\end{lemma}	

\begin{proof}
	\begin{itemize}
		\item[(i)] This follows from $(\gamma - A^{\star})^{-1} = J (\gamma - A)^{-1} J^{\star}$ for $\gamma \in \rho(A)$.
		\item[(ii)] For general closed operators $A$ on a Hilbert space, we have that $\sigma(A^{\star})$ equals the complex conjugate of $\sigma(A)$, and similarly for $\sigma_p(A)$ (e.g., \cite[Theorem III.6.22]{Kato1995}). Combining with the above equality used in  the proof of (i) gives the stated result.
		\item[(iii)] This follows from 
		\[ \gamma_1 \langle J \phi_1, \phi_2 \rangle = \langle J A \phi_1, \phi_2 \rangle  = \langle A^{\star} J \phi_1, \phi_2 \rangle = \langle J \phi_1, A \phi_2 \rangle = \overline \gamma_2 \langle J \phi_1, \phi_2 \rangle.\]	
		\item[(iv)] If $\phi$ is an eigenvector of $A$ with eigenvalue $\gamma$, then $A^{\star} J \phi = J A \phi = \gamma J \phi$.
		\item[(v)] Let  $\overline \Gamma$ denote a counterclockwise contour in the complex plane around (only) the eigenvalue $\overline \gamma$. By \cite[Theorem III.6.22]{Kato1995}, the adjoint of the corresponding spectral projection satisfies 
		\[ (P_{\gamma})^{\star} = \frac 1 { 2 \pi \i} \int_{\overline \Gamma} (\zeta - A^{\star})^{-1} \, d \zeta = \frac 1 {2 \pi \i} J \int_{\overline \Gamma} (\zeta - A)^{-1} \, d \zeta J^{\star} = J P_{\overline \gamma}J^{\star}.\] 
	\end{itemize}
\end{proof}

For later reference we note the following general observation for eigenfunctions of linear operators on function spaces.
\begin{lemma}
	\label{lem:general-result-eigenfunctions}
	Suppose $(A, \mathcal D(A))$ is a closed operator on $L^2(\mu)$. Suppose 
	\[ \psi \in \mathcal D(A) \Longrightarrow \Re \psi \in \mathcal D(A) \quad \text{and}\quad \Im \psi \in \mathcal D(A),\] and suppose $A$ maps (a.e.) real valued functions in its domain into (a.e.) real valued functions.
	Suppose $\phi$ is an eigenfunction of $A$ with eigenvalue $\gamma$. Then $\overline \phi$ is an eigenfunction of $A$ with eigenvalue $\overline \gamma$.
\end{lemma}
\begin{proof}
	By linearity of $A$,
	\[ \overline{A \overline \phi} = \overline {A \Re \phi- \i  A \Im \phi}  = A \Re \phi + \i A \Im \phi =  A \phi = \gamma \phi,\]
	which establishes the result.
\end{proof}

We are now ready to state the main result of this section, which characterizes $\sigma(L)$.
Define functions $\psi^{\pm}(\gamma) := \int_0^{\pm \infty} U'(\xi) e^{\mp 2 \gamma \xi - U(\xi)} \, d \xi$, i.e.,
\begin{equation} \label{eq:psi+-} \psi^{+}(\gamma) := \int_0^{\infty} U'(\xi) e^{- 2 \gamma \xi - U(\xi)} \, d \xi \quad \text{and} \quad \psi^-(\gamma) := -\int_{-\infty}^0 U'(\xi) e^{2 \gamma \xi - U(\xi)}\, d \xi, \quad \gamma \in \C, \end{equation}
and 
\begin{equation} \label{eq:Z} Z(\gamma) := 1 - \psi^+(\gamma) \psi^-(\gamma), \quad \gamma \in \C.\end{equation}
Define the set
\begin{equation}
\label{eq:eigenvalues} \Sigma := \{ \gamma \in \C:  Z(\gamma) = 0 \}.
\end{equation}
Under Assumption~\ref{ass:compactness} the functions $\psi^{\pm}$ are well defined, using Lemma~\ref{lem:integrals}.
\begin{theorem}[Spectrum, resolvent and eigenfunctions of the zigzag generator]
	\label{thm:spectrum}
	Suppose Assumptions~\ref{ass:sobolev}, \ref{ass:compactness} and \ref{ass:unimodal} are satisfied. Furthermore assume $\lambda_{\refr}(x) = 0$ for all $x$. Then $\sigma(L) = \sigma_p(L) = \Sigma$.
	
	If $\gamma \in \rho(L)$ and $h \in L^2(\mu)$ then $f = (\gamma -L)^{-1} h$ is given by
	\begin{align} \label{eq:resolvent}
	\nonumber f^+(x) & = \begin{cases} e^{\gamma x} \left(k^+ + \int_x^0 e^{-\gamma \xi} h^+(\xi) \, d \xi \right), \quad & x \leq 0, \\e^{\gamma x + U(x)} \left(k^+ - \int_0^x e^{-\gamma \xi - U(\xi)} \left[ h^+(\xi) + U'(\xi) e^{-\gamma \xi} \left(k^- + \int_0^{\xi} e^{\gamma \eta} h^-(\eta) \, d \eta \right) \right] \, d \xi \right), \quad & x > 0,\end{cases} \\
	f^-(x) &= \begin{cases} e^{-\gamma x + U(x)} \left( k^- - \int_x^0 e^{\gamma \xi - U(\xi)} \left[ h^-(\xi) - U'(\xi) e^{\gamma \xi} \left( k^+ + \int_{\xi}^0 e^{-\gamma \eta} h^+(\eta) \, d \eta \right) \right] \, d \xi \right), \quad & x \leq 0, \\
	e^{-\gamma x} \left( k^- + \int_0^x e^{\gamma \xi} h^-(\xi) \, d \xi \right), \quad & x > 0.\end{cases}
	\end{align}
	The constants $k^{\pm} = k^{\pm}(\gamma; h)$ are given by
	\begin{equation} \label{eq:constants-k}
	\begin{pmatrix} k^+(\gamma;h) \\ k^-(\gamma;h) \end{pmatrix}= \frac 1 {Z(\gamma)} \begin{pmatrix} 1 & \psi^+(\gamma) \\ \psi^-(\gamma) & 1 \end{pmatrix}K(\gamma) h,
	\end{equation}
	with $K(\gamma) : L^2(\mu) \rightarrow \C^2$ the bounded linear mapping  given by
	\begin{equation} \label{eq:integral-operator} K(\gamma) h = \begin{pmatrix} \int_0^{\infty} e^{-\gamma \xi - U(\xi)} h^+ (\xi) \, d \xi + \int_0^{\infty} U'(\xi) e^{-2 \gamma \xi - U(\xi)} \int_0^{\xi} e^{\gamma \eta} h^-(\eta) \, d \eta \,  d \xi \\
	\int_{-\infty}^0 e^{\gamma \xi - U(\xi)} h^-(\xi) \, d \xi - \int_{-\infty}^0 U'(\xi) e^{2 \gamma \xi - U(\xi)} \int_{\xi}^0 e^{-\gamma \eta} h^+(\eta) \, d \eta \, d \xi \end{pmatrix}.\end{equation}
	
	For every $\gamma \in \sigma_p(L)$, the corresponding eigenfunctions form a one-dimensional space spanned by the function $f_{\gamma} \in \mathcal D(L)$ defined by
	\begin{align*}
	f^+_{\gamma}(x) & = \psi^+(\gamma) e^{\gamma x}, \quad (x \leq 0), & \quad f^+_{\gamma}(x) & = e^{\gamma x + U(x)}  \int_x^{\infty} U'(\xi) e^{-2 \gamma \xi - U(\xi)} \, d \xi, \quad (x \geq 0), \\
	\quad f^-_{\gamma}(x) & = e^{-\gamma x}, \quad (x \geq 0),& \quad  f^-_{\gamma}(x) &= - \psi^+(\gamma)e^{-\gamma x + U(x)} \int_{-\infty}^x U'(\xi) e^{2 \gamma \xi - U(\xi)} \, d \xi, \quad (x \leq 0).
	\end{align*}
	
	If $\gamma_1, \gamma_2\in \sigma_p(L)$, $\gamma_1 \neq \gamma_2$, with corresponding eigenfunctions $f_1$, $f_2$ then either $\gamma_1 = \overline \gamma_2$ or $\langle F f_1, f_2 \rangle = 0$.
\end{theorem}


\begin{proof}
	
	\emph{Claim: $\Sigma \subset \sigma(L)$}.
	
	\emph{Proof of claim}.
	We will prove the equivalent statement that $\rho(L) \subset \C\setminus \Sigma$.
	Suppose $\gamma \in \rho(L)$. In particular, for all $h \in L^2(\mu)$, there is a solution $f \in \mathcal D(L)$ to the equation $(\gamma - L) f = h$.$f$ satisfies the coupled system of ordinary differential equations
	\begin{align*}
	\gamma f^+ - \partial_x f^+ - \lambda^+ (f^- - f^+) & = h^+, \\
	\gamma f^- + \partial_x f^- - \lambda^- (f^+ - f^-) & = h^-.
	\end{align*}
	For $x \leq 0$, we have $\lambda^+(x) = 0$ and $\lambda^-(x) = -U'(x)$. We find that 
	\[ f^+(x) = e^{\gamma x} \left( k^+  + \int_{x}^0 e^{-\gamma \xi} h^+(\xi) \, d \xi \right), \quad x \leq 0, \] for some constant $k^+$.
	By variation of constants, for $x \leq 0$,
	\begin{align*} f^-(x) & = e^{-\gamma x + U(x)} \left( k^- - \int_{x}^0 e^{\gamma \xi - U(\xi)} \left[ h^-(\xi) + \lambda^-(\xi) f^+(\xi)\right] \, d \xi \right) \\
	& =  e^{-\gamma x + U(x)} \left( k^- - \int_{x}^0 e^{\gamma \xi - U(\xi)} \left[ h^-(\xi) -U'(\xi) e^{\gamma \xi} \left(k^+ + \int_{\xi}^0 e^{-\gamma \eta} h^+(\eta)\, d \eta \right)\right] \, d \xi \right).
	\end{align*}
	Similarly, for $x \geq 0$, we obtain
	\begin{align*}
	f^-(x) & = e^{-\gamma x} \left( k^- + \int_0^x e^{\gamma \xi} h^-(\xi) \, d \xi \right), \\
	f^+(x) & = e^{\gamma x + U(x)} \left(k^+ - \int_0^x e^{-\gamma \xi - U(\xi)}[ h^+(\xi) + U'(\xi) f^-(\xi)] \, d \xi\right) \\
	& = e^{\gamma x + U(x)} \left(k^+ - \int_0^x e^{-\gamma \xi - U(\xi)}\left[ h^+(\xi) + U'(\xi) e^{-\gamma \xi} \left( k^- + \int_0^\xi e^{\gamma \eta} h^-(\eta) \, d \eta \right)\right] \, d \xi \right).
	\end{align*}
	Here no new integration constants are introduced because we required continuity of $f$. Indeed by assumption $f \in \mathcal D(L)$ and thus $f$ is  continuous.
	
	By Lemma~\ref{lem:integrals}, we have that all the integrals occuring in the expressions for $f^+(x)$ for $x \geq 0$ and (analogously) $f^-(x)$ for $x \leq 0$, converge as $|x| \rightarrow \infty$. In order for $f \in L^2(\mu)$, it follows that necessarily
		\begin{align} \nonumber  k^+ - \int_0^{\infty} e^{-\gamma \xi - U(\xi)} \left[h^+(\xi) + U'(\xi) e^{-\gamma \xi} \left( k^- + \int_0^\xi e^{\gamma \eta} h^-(\eta) \, d \eta \right) \right] \, d \xi & =0 \quad & \text{and} \\
	\label{eq:relations-k+-} k^- -  \int_{-\infty}^0 e^{\gamma \xi - U(\xi)} \left[ h^-(\xi) -U'(\xi) e^{\gamma \xi} \left(k^+ + \int_{\xi}^0 e^{-\gamma \eta} h^+(\eta)\, d \eta \right)\right] \, d \xi & = 0,
	\end{align}
for otherwise the multiplication by $e^{\pm \gamma x + U(x)}$ in the expression for $f$ would lead to the norm of $f$ being equal to $\infty$. (Note that  the exponential growth by $e^{\pm \Re \gamma x}$ in $f^+$ for $x  \geq 0$ and $f^-$ for $x \leq 0$  is not problematic since it is compensated by $e^{-U(x)}$ in the $L^2(\mu)$ inner product, which by Assumption~\ref{ass:compactness} will decay faster than any exponential.)
	Condition~\eqref{eq:relations-k+-} can be written in terms of a linear system for $k^{\pm}$:
	\begin{equation} \label{eq:system-k+-} \begin{pmatrix} 1 & - \psi^+(\gamma) \\ -\psi^-(\gamma) & 1 \end{pmatrix} \begin{pmatrix} k+ \\ k^- \end{pmatrix} = K(\gamma) h \end{equation}
	where, for $\gamma \in \C$,  $K(\gamma) : L^2(\mu) \rightarrow \C^2$ is defined by~\eqref{eq:integral-operator}. The boundedness of $K$ follows by the proof of  Lemma~\ref{lem:integrals}.
	
	Now suppose $\gamma \in \Sigma$. In this case the system~\eqref{eq:system-k+-} is singular. However the linear map $K(\gamma) : L^2(\mu) \rightarrow \C^2$ is surjective, even when we restrict the domain to the functions in $L^2(\mu)$ having compact support. 
	Indeed, we can produce the vector $\begin{pmatrix}1 & 0\end{pmatrix}^T$ on the right hand side by choosing 
	\[ h^-(x) = 0 \quad \text{and} \quad h^+(x) = e^{\gamma x + U(x)} \1_{[0,1]}(x), \quad x \in \R,\]
	and in an analogous way we can choose $h$ so that the vector $\begin{pmatrix}0 & 1 \end{pmatrix}^T$ is obtained on the right hand side. Thus the right hand side as a function of $h$ having compact support has range $\C^2$, whereas the left hand side can only span a one-dimensional subspace due to the assumption $\gamma \in \Sigma$. 
	It follows that there are choices $h$ such that no solution to the resolvent system $(\gamma  - L) f = h$ exist, so that $\gamma \notin \rho(L)$; a contradiction.
	The solution of~\eqref{eq:system-k+-} yields the stated expression for $k^{\pm}$.
	
	
	\emph{Claim: $\sigma_p(L) \subset \Sigma$}.
	
	\emph{Proof of claim}.
	Suppose $\gamma \in \sigma_p(L)$. There exists a function $f \in \mathcal D(L)$ such that $L f = \gamma f$. This corresponds to the system
	\begin{align*}
	\gamma f^+ - \partial_x f^+ - \lambda^+ (f^- - f^+) & = 0, \\
	\gamma f^- + \partial_x f^- - \lambda^- (f^+ - f^-) & = 0.
	\end{align*}
	Solving for $f^+$ for $x \leq 0$ yields $f^+(x) = c^+ e^{\gamma x}$, and similarly $f^-(x) = c^- e^{-\gamma x}$ for $x \geq 0$ for some constants $c^{\pm} \in \C$. Then using variation of constants and insisting upon continuity in $x = 0$, yields
	\begin{align*}
	f^+(x) & = e^{\gamma x + U(x)} \left(c^+ - c^- \int_0^x U'(\xi) e^{-2 \gamma \xi - U(\xi)} \, d \xi \right), \quad & x \geq 0, \\
	f^-(x) & = e^{-\gamma x + U(x)} \left(c^- + c^+ \int_x^0 U'(\xi) e^{2 \gamma \xi - U(\xi)} \, d \xi \right), \quad & x \leq 0.
	\end{align*}
	Again the integrals converge as $|x| \rightarrow \infty$ by Lemma~\ref{lem:integrals}.  Now in order for $f \in L^2(\mu)$, we require that the terms within parantheses in the expressions for $f^{\pm}$ vanish as $|x| \rightarrow \infty$. This yields the conditions
	\begin{equation} \label{eq:eigenfunctions-constants} 
	c^+ = c^- \psi^+(\gamma), \quad  \text{and}  \quad 
	c^- = c^+ \psi^-(\gamma).
	\end{equation}
	This system admits a non-trivial solution if and only if $\gamma \in \Sigma$. The choice for $f$ in the statement of the proposition is obtained by taking $c^- = 1$ and $c^+ = \psi^+(\gamma)$.
	
	We have established that $\Sigma \subset \sigma(L)$ and $\sigma_p(L) \subset \Sigma$. Since $L$ has compact resolvent (Corollary~\ref{cor:compact-resolvent}), $\sigma_p(L) = \sigma(L)$ and the proof is complete.
	%
	
	The final statement of the theorem follows by an application of Lemma~\ref{lem:J-selfadjoint} (iii).
\end{proof}

\begin{corollary}
	\label{cor:spectral-projection}
	The spectral projection corresponding to $\gamma \in \sigma(L)$ is given by
	\begin{equation} \label{eq:spectral-projection} P_{\gamma} = \frac 1 {2 \pi \i} \int_{\Gamma} \frac 1 { Z(\zeta)} M(\zeta) \begin{pmatrix}1 & \psi^+(\zeta) \\ \psi^-(\zeta) & 1 \end{pmatrix} K(\zeta) \, d \zeta,  \end{equation}
	where $\Gamma$ is a Jordan contour in $\C$ enclosing the eigenvalue $\gamma$ and no other eigenvalues, and where the family of bounded linear operators $(M(\zeta))_{\zeta \in \rho(L)}$ is given for fixed $\zeta \in \rho(L)$ as a bounded operator $M(\zeta) : \C^2 \rightarrow L^2(\mu)$  by 
	\begin{align*} (M(\zeta)a)^+(x) & := M^+(\zeta) a (x)  := \begin{cases} e^{\zeta x} a^+, \quad & x \leq 0, \\
	e^{\zeta x + U(x)} \left( a^+ - a^- \int_0^x U'(\xi) e^{-2 \zeta \xi - U(\xi)} \, d \xi \right),\quad & x > 0, \end{cases} \\
	(M(\zeta)a)^-(x) & := M^-(\zeta) a(x) := \begin{cases} e^{-\zeta x + U(x)} \left(a^- + a^+ \int_x^0 U'(\xi) e^{2 \zeta x - U(\xi)} \, d \xi \right), \quad & x \leq 0, \\
	e^{-\zeta x} a^-,  \quad & x > 0,
	\end{cases}
	\end{align*}
	for $a = \begin{pmatrix} a^+ \\ a^- \end{pmatrix} \in \C^2$.
\end{corollary}
Note that the integral in~\eqref{eq:spectral-projection} is a concatenation of bounded linear operators and as such the integrand is a well defined mapping from $\rho(L)$ into the bounded operators on $L^2(\mu)$, which makes the integral over closed contours well defined as a Bochner integral.

\begin{proof}
	By \cite[proof of Theorem 6.17]{Kato1995}, the spectral projections are given for $\gamma \in \sigma(L)$ by
	\[ P_{\gamma}  = \frac 1 {2 \pi \i} \int_{\Gamma} (\zeta - L)^{-1} \, d \zeta,\]
	where $\Gamma$ is a contour enclosing $\gamma$ as stated.
	Note that the terms in the resolvent expression~\eqref{eq:resolvent} which do not involve $k^{\pm}$ amount to integral operators by an analytic integral kernel. We may check that such terms are holomorphic as operator valued functions of the complex-valued argument. Indeed, as discussed in \cite[Section VII.1.1]{Kato1995} weak and strong notions of holomorphicity coincide, and we may check that for example the operator $T^+(\zeta)$, constituting the non-$k^{\pm}$-dependent terms in $((\zeta -L)^{-1}h)^+$, given by 
	\[ T^+(\zeta) h(x) = \begin{cases} e^{\zeta x} \int_x^0 e^{-\zeta \xi} h^+(\xi) \, d \xi, \quad &x \leq 0, \\
	                 e^{\zeta x + U(x)} \left( - \int_0^x e^{-\zeta \xi - U(\xi)} \left[h^+(\xi) + U'(\xi) e^{-\zeta \xi} \int_0^{\xi} e^{\zeta \eta} h^-(\eta) \, d \eta \right] \right), \quad & x > 0,
	                \end{cases} \]
	                is holomorphic as a function of $\zeta$ by verifying that $\langle g, T^+(\zeta) h \rangle_{L^2(\nu)}$ is holomorphic for indicator functions $g$, $h$. Analogously this can be verified for $((\zeta -L)^{-1}h)^-$.

	Thus, if we perform the complex contour integration all terms in~\eqref{eq:resolvent} which do not involve $k^{\pm}$ vanish, since such terms are holomorphic in $\zeta$. The constants $k^{\pm}$ are given by~\eqref{eq:constants-k}.
	Combining these observations yields the stated expression.
\end{proof}

\begin{corollary}[Simple roots] 
	\label{cor:simple-root}
	Suppose the assumptions of Theorem~\ref{thm:spectrum} are satisfied and $\gamma$ is a simple root of $Z$, so that $\gamma \in \sigma(L)$ and $Z'(\gamma) \neq 0$. The corresponding spectral projection has rank one and is given by
	\[ P_{\gamma} h =  \frac{ \begin{pmatrix} 1  & \psi^+(\gamma) \end{pmatrix}K(\gamma) h}{Z'(\gamma) \psi^+(\gamma)} f_{\gamma}\]
	where $f_{\gamma}$ is the eigenfunction corresponding to $\gamma$ as introduced in Theorem~\ref{thm:spectrum}.
	An alternative expression for $P_{\gamma}$ is given by
	\[ P_{\gamma} h = \frac{\langle h, F \overline f_{\gamma} \rangle}{\langle f_{\gamma}, F \overline f_{\gamma} \rangle} f_{\gamma}.\]
\end{corollary}

\begin{proof}
	Note that $Z(\gamma) = 0$ yields $\psi^-(\gamma) = 1/\psi^+(\gamma)$. The Cauchy residue theorem applied to the expression for $P_{\gamma}$ of Corollary~\ref{cor:spectral-projection} yields that
	\[ P_{\gamma} = \frac 1 {Z'(\gamma)} M(\gamma) \begin{pmatrix} 1 \\  \psi^-(\gamma) \end{pmatrix} \begin{pmatrix}1 & \psi^+(\gamma) \end{pmatrix} K(\gamma).\]
	Next we compute
	\begin{align*}
	M^+(\gamma) \begin{pmatrix}1 \\ \psi^-(\gamma) \end{pmatrix}(x)& = \begin{cases} e^{\gamma x}, \quad & x \leq 0, \\
	e^{\gamma x + U(x)} \left(1 - \psi^-  \int_0^x U'(\xi) e^{-2 \gamma \xi - U(\xi)} \, d \xi \right),\quad & x > 0, \end{cases} \\
	& =  \begin{cases} e^{\gamma x}, \quad & x \leq 0, \\
	\psi^-(\gamma) e^{\gamma x + U(x)}\int_x^{\infty} U'(\xi) e^{-2 \gamma \xi - U(\xi)} \, d \xi,\quad & x > 0, \end{cases} 
	\end{align*}
	and similarly
	\[ M^-(\gamma) \begin{pmatrix} 1 \\ \psi^-(\gamma) \end{pmatrix}(x) = \begin{cases}  - e^{-\gamma x + U(x)} \int_{-\infty}^x  U'(\xi) e^{2 \gamma x - U(\xi)} \, d \xi, \quad & x \leq 0, \\
	\psi^-(\gamma)e^{-\gamma x},  \quad & x > 0.
	\end{cases}\]
	Multiplying by $\frac{\psi^+(\gamma)}{\psi^+(\gamma)}$ and using that $\psi^+(\gamma) \psi^-(\gamma) = 1$ yields the stated result.
	
	Furthermore since $P_{\gamma}$ is of rank one and has range $\operatorname{span} \{f_{\gamma}\}$, there exists a $g \in L^2(\mu)$ such that $P_{\gamma} h = \langle h, g \rangle f_{\gamma}$.  
	We have $P_{\gamma}^{\star} m = \langle m, f_{\gamma} \rangle g$, for $m \in L^2(\mu)$.
	Since $P_{\gamma}^{\star}$ is the spectral projection corresponding to eigenvalue $\overline \gamma$ of $L^{\star}$ (see \cite[Theorem III.6.22]{Kato1995}), it follows that $g$ is an eigenfunction of $L^{\star}$ with eigenvalue $\overline \gamma$. By Lemma~\ref{lem:general-result-eigenfunctions}, $\overline g$ is an eigenfunction of $L^{\star}$ with eigenvalue $\gamma$, and by Lemma~\ref{lem:J-selfadjoint} (iv), $F \overline g$ is an eigenfunction of $L$ with eigenvalue $\gamma$, i.e., $F \overline g$ is parallel to $f_{\gamma}$. It follows that $F \overline g = \alpha f_{\gamma}$ for some $\alpha \in \C$, i.e., $g = \overline{\alpha} F \overline {f_{\gamma}}$. From $P_{\gamma}^2 = P_{\gamma}$ it follows that $\langle f_{\gamma}, g \rangle = 1$ which yields the correct value for $\alpha$.
\end{proof}
%

\begin{remark} \label{rem:expression-Z-prime}
	Corollary~\ref{cor:simple-root} yields an expression for $Z'(\gamma)$ for $\gamma \in \sigma(L)$, for $Z'(\gamma)\neq 0$:
	\[ Z'(\gamma) = \frac{\langle f_{\gamma}, F \overline f_{\gamma} \rangle}{\langle h, F \overline f_{\gamma} \rangle}\frac{\begin{pmatrix} 1  & \psi^+(\gamma) \end{pmatrix}K(\gamma) h}{\psi^+(\gamma)}.\] It may be computed directly that $\begin{pmatrix} 1  & \psi^+(\gamma) \end{pmatrix}K(\gamma) h = \langle h, F \overline f_{\gamma} \rangle$, so that 
	\[Z'(\gamma) = \langle f_{\gamma}, F \overline f_{\gamma} \rangle/\psi^+(\gamma) = \psi^-(\gamma) \langle f_{\gamma}, F \overline f_{\gamma} \rangle.\] In fact this expression for $Z'(\gamma)$ remains valid if $Z'(\gamma) = 0$, which may be verified by direct computation, showing that $\langle f_{\gamma}, F \overline f_{\gamma} \rangle = 0$, and making using of the following lemma. 
\end{remark}

\begin{lemma}
	\label{lem:psi}
	Suppose the assumptions of Theorem~\ref{thm:spectrum} hold.  Then the functions $\psi^{\pm}$ (as defined in~\eqref{eq:psi+-}) are holomorphic, and satisfy
	\begin{equation} \label{eq:psi_alternative} \psi^{\pm}(\gamma) = 1 \mp 2 \gamma \int_0^{\pm \infty} e^{\mp 2 \gamma \xi - U(\xi)}\, d \xi.\end{equation}
	Furthermore
	\begin{equation} \label{eq:dpsi} \frac{d}{d \gamma} \psi^{\pm}(\gamma) = \mp 2 \int_0^{\pm \infty} e^{\mp 2 \gamma \xi - U(\xi)} \, d \xi + 4 \gamma \int_0^{\pm \infty} \xi e^{\mp 2 \gamma \xi - U(\xi)} \, d \xi, \quad \gamma \in \C.\end{equation}
\end{lemma}
\begin{proof}
	This follows by partial integration.
\end{proof}

\begin{proposition}[Dominant eigenvalue]
	\label{prop:dominant-eigenvalue}
	Suppose the assumptions of Theorem~\ref{thm:spectrum} hold. Then 0 is a simple eigenvalue of $L$ (i.e. its spectral projection has rank one). The spectral projection maps onto the space of constant functions.
\end{proposition}
\begin{proof}
	From~\eqref{eq:psi_alternative}, $\psi^{\pm}(0) = 1$, so that $Z(0) = 1 - \psi^+(0) \psi^-(0) = 0$.
	Furthermore, using~\eqref{eq:dpsi}, 
	\[ \frac{d}{d \gamma} Z(0) = - \psi^+(0) \frac{d}{d \gamma} \psi^-(0) - \psi^-(0) \frac{d}{d \gamma} \psi^+(0)= 2 \int_{-\infty}^{\infty} e^{-U(\xi)} \, d \xi \neq 0.\]
	The range of the spectral projection coincides with the span of the eigenfunction corresponding to the eigenvalue 0, which is constant  by Theorem~\ref{thm:spectrum}.
\end{proof}

Define the closed subspace $L^2_0(\mu) := \{ f \in L^2(\mu) : \mu(f) = 0 \} = \1^{\perp}$. From Proposition~\ref{prop:dominant-eigenvalue} and Corollary~\ref{cor:simple-root} it follows that $L^2_0(\mu) = (I-P_0) L^2(\mu)$, where $P_0 = \langle \cdot, \1 \rangle / \langle \1, \1 \rangle \1$ denotes the spectral projection corresponding to the eigenvalue 0. Since spectral projections associated with $L$ commute with the semigroup generated by $L$, it follows that the semigroup $(P(t))_{t \geq 0}$ leaves $L^2_0(\mu)$ invariant. The restriction of $(P(t))_{t \geq 0}$ to $L^2_0(\mu)$ has generator $L_0$, defined to be the restriction of $L$ to the domain $\mathcal D(L_0) = \mathcal D(L) \cap L^2_0(\mu)$. This establishes the following result.

\begin{proposition}[Poisson equation]
	\label{prop:poisson-equation}
	Suppose the assumptions of Theorem~\ref{thm:spectrum} are satisfied. Then $\sigma(L_0) = \sigma_p(L_0) = \sigma(L) \setminus\{0\}$. In particular $L_0$ admits a bounded inverse.
\end{proposition}

\subsection{Spectral gap}

The following theorem establishes that the non-trivial spectrum of $L$ is strictly separated from the imaginary axis. The proof depends upon Lemmas~\ref{lem:no-finite-accumulation-point} and~\ref{lem:no-infinite-accumulation-point}, which may be found in the appendix.

\begin{theorem}[Spectral gap]
\label{thm:spectralgap}
Suppose the assumptions of Theorem~\ref{thm:spectrum} are satisfied. Define 
\[ \kappa := \inf \{-\Re \gamma : \gamma \in \sigma(L), \gamma \neq 0\}.\]
Then $\kappa > 0$.
\end{theorem}

\begin{proof}
Suppose on the contrary there exists a sequence $(\gamma_n) \in \sigma(L) \setminus \{0\}$ such that  $\Re \gamma_n \rightarrow 0$. By Lemma~\ref{lem:no-finite-accumulation-point}, it follows that $\Im \gamma_n$ is unbounded (for otherwise there would be an accumulation point $\i \beta$ for $\beta \in \R$). By extracting a subsequence, we may assume that $|\Im \gamma_n | \rightarrow \infty$. However, by Lemma~\ref{lem:no-infinite-accumulation-point}, applied to a small interval $\alpha \in [-\varepsilon, 0]$ this is impossible.
\end{proof}

\subsection{A spectral mapping theorem}

In this section we will establish a spectral mapping theorem, which maps the spectrum of the infinitesemal generator to the spectrum of the semigroup. If the semigroup were eventually compact, this would be immediate (see, e.g., \cite[Corollary IV.3.12]{EngelNagel2000}); however this is not the case as we have established in Proposition~\ref{prop:not-eventually-compact}. 

By the Gearhart-Herbst theorem \cite[Theorem 4.5]{Gearhart1978}  spectrum of the semigroup may also originate from vertical lines in the complex plane (i.e. $\Re \gamma = \text{constant}$) along which the resolvent of the generator is unbounded. We verify in Proposition~\ref{prop:resolvent-bound} that this does not happen if we include Assumption~\ref{ass:polynomial-tail}, which details a sufficiently fast decay of the stationary distribution in its tails. Using Proposition~\ref{prop:resolvent-bound} we can then prove the spectral mapping theorem, Theorem~\ref{thm:spectralmapping}. 
%
%

\begin{proposition}
	\label{prop:resolvent-bound}
	Suppose Assumptions~\ref{ass:sobolev}, \ref{ass:compactness}, \ref{ass:unimodal} and~\ref{ass:polynomial-tail} are satisfied. Furthermore assume $\lambda_{\refr}(x) = 0$ for all $x$.  There is a family of constants $C(\alpha)$ such that for all $\alpha \in \R$ we have
	\[ \limsup_{|\beta| \rightarrow \infty} \|(\alpha + \i \beta - L)^{-1}\| \leq C(\alpha).\]
\end{proposition}

\begin{proof}
	Suppose $\gamma \in \rho(L)$ and write $\gamma = \alpha + \i \beta$, where $\alpha, \beta \in \R$. In this proof we will often write $C(\alpha)$ for a positive constant which depends only on $\alpha$, whose value may be different at different locations in the proof.
	Let $h \in L^2(\mu)$ and let $f = (\gamma - L)^{-1} h$.
	The proof of Theorem~\ref{thm:spectrum} gives an expression for $f$ in terms of $h$. It only remains to provide adequate bounds for the expressions.
	
	We have by the proof of Theorem~\ref{thm:spectrum}, for $x \leq 0$, that 
	\[ f^+(x) = k^+ e^{\gamma x}  + e^{\gamma x} \int_x^0 e^{-\gamma \xi} h^+(\xi) d \xi.\]
	We will first the following claim.
	
	\emph{Claim 1:} $\|x \mapsto f^+(x) \1_{\{x \leq 0\}}\|_{L^2(\nu)} \leq C(\alpha) \left(|k^+| + \|h^+\|_{L^2(\nu)}\right)$.
	
	\emph{Proof of Claim 1.}
	For the first term of $f^+$, $\int_{-\infty}^0 |e^{2 \gamma x}| e^{-U(x)} \, d x = \int_{-\infty}^0 e^{2 \alpha x} e^{-U(x)} \, d x  \leq C(\alpha)$ by Assumption~\ref{ass:compactness}.
	For the second term of $f^+$, by Lemma~\ref{lem:convolution}, taking $\psi(y) = h^+(-y)\1_{\{y \geq 0\}}$ and $\phi(y) = e^{-\gamma y}\1_{\{y \geq 0\}}$, we find that (substituting $x = - y$ and $\xi = -\eta$, and noting that Assumption~\ref{ass:polynomial-tail} is invariant under this reparametrization)
	\begin{align*} \left\| x \mapsto  e^{\gamma x}\int_x^0 e^{-\gamma \xi} h^+(\xi) d \xi  \, \1_{\{x \leq 0\}} \right\|_{L^2(\nu)}  & = \left\| y \mapsto \int_0^y e^{-\gamma (y - \eta)} h^+(-\eta) \, d \eta\, \1_{\{y \geq 0\}} \right\|_{L^2(\tilde \nu)}\\
	& = \| \phi \star \psi\|_{L^2(\tilde \nu)} \\
	& \leq \left( \int_0^{\infty} e^{-\alpha x - C/2 - (m/2)|x|^p} \, d x \right) \|h^+\|_{L^2(\nu)},
	\end{align*} 
	where $\tilde \nu(dy) = e^{-U(-y)} \, d y$.
	This establishes Claim 1.
	
	
	By an analogous argument, we have that $\|x \mapsto f^-(x) \1_{\{x \geq 0\}}\|_{L^2(\nu)} \leq C(\alpha) \left(|k^-| + \|h^-\|_{L^2(\nu)}\right)$.

	Next we will obtain a similar estimate for
	\[ f^+(x) = e^{\gamma x + U(x)} \left(k^+ - \int_0^x e^{-\gamma \xi - U(\xi)} \left[ h^+(\xi) + U'(\xi) e^{-\gamma \xi} \left(k^- + \int_0^{\xi} e^{\gamma \eta} h^-(\eta) \, d \eta \right) \right]  \, d \xi \right), \quad x \geq 0.\]

	\emph{Claim 2:} $\|x \mapsto f^+(x) \1_{\{x \geq 0\}}\|_{L^2(\nu)} \leq C(\alpha) \left(|k^-| + \| h^+\|_{L^2(\nu)} + \| h^-\|_{L^2(\nu)} \right)$. 
	
	\emph{Proof of Claim 2.}
	Recall from the proof of Theorem~\ref{thm:spectrum} that $k^+$ and $k^-$ satisfy the relations~\eqref{eq:relations-k+-}. By substituting the expression for $k^+$ in the above expression for $f^+$ we obtain
	\[  f^+(x) = e^{\gamma x + U(x)} \int_x^{\infty} e^{-\gamma \xi - U(\xi)} \left[ \widetilde h^+(\xi) + k^- U'(\xi) e^{-\gamma \xi} \right]\, d \xi, \quad x \geq 0.\]
	Here $\widetilde h^+(\xi) = h^+(\xi) + U'(\xi) e^{-\gamma \xi} \int_0^{\xi} e^{\gamma \eta} h^-(\eta) \, d \eta $, $\xi \geq 0$. First we will establish \emph{Claim 2a:}  $x \mapsto \widetilde h^+ (x)\1_{\{x \geq 0\}} \in L^2(\nu)$, with norm depending only on $\|h^+\|_{L^2(\nu)}$, $\|h^-\|_{L^2(\nu)}$ and $\Re \gamma$. We only need  to establish this for the second term of $\widetilde h^+$. 
	We may estimate
	\begin{align*}
	\left|e^{-\gamma \xi} \int_0^{\xi} e^{\gamma \eta} h^-(\eta) \, d \eta \right|\leq e^{-\alpha \xi} \int_0^{\xi} e^{\alpha \eta} |h^-(\eta)| \, d \eta.
	\end{align*}
	An analogous argument as given in the proof of Claim 1 establishes that 
	\[ \phi : x \mapsto e^{-\alpha x} \int_0^x e^{\alpha \eta} |h^-(\eta)| \, d \eta \1_{\{x \geq 0\}} \in L^2(\nu),\] with $\|\phi\|_{L^2(\nu)}$ depending only on $\alpha$ and $\|h^-\|_{L^2(\nu)}$.
	Then $\phi$ has weak derivative $\partial_x \phi(x) = \left( |h^-(x)| - \alpha \phi(x) \right) \1_{\{x \geq 0\}}$ in $L^2(\nu)$ which may be bounded by a constant depending only on $\|h^-\|_{L^2(\nu)}$ and $\alpha$. Thus $\| \phi \|_{W^{1,2}(\nu)} \leq C(\alpha) \|h^-\|_{L^2(\nu)}$. By Lemma~\ref{lem:sobolev-classical}, $U'(x) \phi(x) \in L^2(\nu)$ with norm depending only on $\alpha$ and $h$. This establishes Claim 2a.
	Next we establish \emph{Claim 2b:} $\xi \mapsto U'(\xi) e^{-\gamma \xi} \in L^2(\nu)$ with norm depending only on $\alpha$. Indeed
	\[ \int_0^{\infty} (U'(\xi))^2 e^{-2 \alpha \xi - U(\xi)} \, d \xi < \infty \] by Lemma~\ref{lem:sobolev-classical}. Combining Claims 2a and 2b yields that for $\psi(\xi) = \widetilde h^+(\xi) + k^- U'(\xi)e^{-\gamma \xi}$, 
	\[ \|\psi \|_{L^2(\nu)} \leq C(\alpha) ( \|h^+\|_{L^2(\nu)} + \|h^-\|_{L^2(\nu)} + |k^-| ).\]
	Now Claim 2 is proven by applying the second statement of Lemma~\ref{lem:convolution}, applied to $\psi$ as defined above and $\phi(x) = e^{-\gamma x}$.
	
	We may repeat the proofs of Claims 1 and 2 to establish analogous results for $f^-$. Thus we obtain that
	\[ \|f \|_{L^2(\mu)} \leq C(\alpha) \left(|k^+| + |k^-| + \|h\|_{L^2(\mu)}\right).\]
	
	We will now express $k^{\pm}$ in terms of $h$ and $\gamma$. 
	From~\eqref{eq:system-k+-} and~\eqref{eq:integral-operator}, we obtain
	\[ \begin{pmatrix} k^+ \\ k^- \end{pmatrix} = \frac 1 {Z(\gamma)} A(\gamma)K(\gamma) h, \quad \text{with} \quad A(\gamma) :=  \begin{pmatrix} 1 &  \psi^+(\gamma) \\ \psi^-(\gamma) & 1 \end{pmatrix}.\]
	The first row of $K(\gamma)$ contains the expression
	\[ \int_0^{\infty} e^{-\gamma \xi - U(\xi)} \widetilde h^+(\xi) \, d \xi,\]
	where $\widetilde h^+ \in L^2(\nu)$ is as in the proof of Claim 2 above.
	In the proof of Claim 2 it is established that $\widetilde h^+ \in L^2(\nu)$, with norm depending only on $\|h\|_{L^2(\mu)}$ and $\alpha = \Re \gamma$. The  same holds for the second row and we conclude that $\| K(\gamma) h \| \leq C(\alpha) \|h\|_{L^2(\mu)}$ for some constant $C(\alpha)$ depending only on $\alpha$.
	
	Recall $\psi^+(\gamma) = \int_0^{\infty} U'(\xi) e^{-2 \gamma \xi - U(\xi)} \, d \xi$. Since under our assumptions $x \mapsto U'(x) e^{-U(x)-2 \alpha x}\1_{\{x \geq 0\}} \in L^1(\R)$ for any $\alpha \in \R$, by Riemann-Lebesgue \cite[Theorem C.8]{EngelNagel2000}, we find for any $\alpha \in \R$ that $\lim_{|\beta| \rightarrow \infty} \psi^+(\alpha + \i \beta) = 0$. By an analogous argument, 
	$\lim_{|\beta| \rightarrow \infty} \psi^-(\alpha + \i \beta) = 0$.
	Thus for any $\alpha$ the matrix $\beta \mapsto A(\alpha + \i \beta)$ is bounded. Furthermore \[ \lim_{|\beta| \rightarrow \infty} Z(\alpha + \i \beta) =  \lim_{|\beta| \rightarrow \infty} \left(1 - \psi^+(\alpha+ \i \beta) \psi^-(\alpha + \i \beta) \right) = 1,\]
	so that the roots of $Z$ are contained in a bounded interval along the line $\Re \gamma = \alpha$.

	Combining all estimates yields the stated result.
\end{proof}

\begin{theorem}[Spectrum of the zigzag semigroup]
	\label{thm:spectralmapping}
	Suppose Assumptions~\ref{ass:sobolev}, \ref{ass:compactness}, \ref{ass:unimodal} and~\ref{ass:polynomial-tail} are satisfied. Furthermore assume $\lambda_{\refr}(x) = 0$ for all $x$. 
	Then for all $t > 0$, 
	\begin{equation} \sigma(P(t)) \setminus \{0\} = \sigma_p(P(t)) \setminus \{0\} = \{ e^{\gamma t} : \gamma \in \sigma(L)\}.\end{equation}
\end{theorem}

\begin{proof}
	Fix $t > 0$.
	Spectral mapping of the point spectrum is well known \cite[Theorem IV.3.7]{EngelNagel2000}: $\{ e^{\gamma t} : \gamma \in \sigma_p(L)\} = \sigma_p(P(t)) \setminus \{0\}$. By Theorem~\ref{thm:spectrum} we have $\sigma_p(L) = \sigma(L)$. Since $\sigma_p(P(t)) \setminus \{0\} \subset \sigma(P(t)) \setminus \{0\}$ it only remains to establish that $\sigma(P(t))  \setminus\{0\} \subset \{ e^{\gamma t} : \gamma \in \sigma(L)\}$.
	Suppose $\eta \in \C$, $\eta \neq 0$, and consider the set 
	\[ \Gamma_{\eta} := \{ \gamma \in \C: e^{\gamma t} = \eta\} = \{ \alpha + \beta i : \alpha = (\ln |\eta|) /t, \beta = (\arg \eta) /t + 2 k \pi/t, k \in \Z\}.\]
	Suppose $\Gamma_{\eta} \cap \sigma(L) = \emptyset$. By Proposition~\ref{prop:resolvent-bound} and the above characterization of $\Gamma_{\eta}$, it follows that $\gamma \mapsto \|(\gamma - L)^{-1}\|$ is bounded on $\Gamma_{\eta}$. Therefore, by \cite[Theorem 4.5]{Gearhart1978}, $\eta \notin \sigma(P(t))$. In other words, if $\eta \in \sigma(P(t)) \setminus \{0\}$, then there is a $\gamma \in \Gamma_{\eta} \cap \sigma(L)$. This establishes that $\sigma(P(t)) \setminus \{0\} \subset \{ e^{\gamma t} : \gamma \in \sigma(L)\}$.
\end{proof}

Recall the spectral gap $\kappa$ of $L$,
\[ \kappa := \inf \{ -\Re \gamma: \gamma  \in \sigma_p(L), \gamma \neq 0\},\]
which by Theorem~\ref{thm:spectralgap} satisfies $\kappa > 0$.
By Assumption~\ref{ass:sobolev}, we have that $\mu(E) < \infty$. Define a Borel probability distribution $\pi$ on $E$ by 
\begin{equation} \label{eq:def-pi} \pi(A) = \mu(A)/\mu(E), \quad \text{for all Borel measurable $A \subset E$.}\end{equation}

\begin{theorem}
	\label{thm:decay}
	Suppose the assumptions of Theorem~\ref{thm:spectralmapping} are satisfied. There is a constant $M > 0$ such that for any $f \in L^2(\mu)$, 
	\[ \|P(t) f - \pi(f)\|_{L^2(\mu)} \leq M e^{-\kappa t}\|f - \pi(f)\|_{L^2(\mu)}.\]
\end{theorem}

For the proof we will consider the spectrum of the semigroup $(P(t))_{t \geq 0}$ restricted to $L_0^2(\mu)$ as defined just before Proposition~\ref{prop:poisson-equation}, with infinitesemal generator $\widetilde L$.

\begin{proof}
	Write $\widetilde P(t)$ for $\left. P(t) \right|_{L^2_0(\mu)}$. The argument of Theorem~\ref{thm:spectralmapping} may be repeated to establish that, for $t > 0$,
	\[ \sigma(\widetilde P(t)) \setminus \{0\} = \sigma_p(\widetilde P(t)) \setminus \{0\} = \{ e^{\gamma t} : \gamma \in \sigma(\widetilde L)\} =  \{ e^{\gamma t} : \gamma \in \sigma(L), \gamma \neq 0\}.\]
	Here we used that $0$ is a simple eigenvalue (Proposition~\ref{prop:dominant-eigenvalue}), which is removed from $\sigma(\widetilde L)$, since $L^2_0(\mu)$ is by definition orthogonal to the associated eigenspace of constant functions.
	It then follows from \cite[Proposition IV.2.2]{EngelNagel2000} that there exists an $M > 0$ such that for the operator norm we have $\|P_0(t)\|_{L^2_0(\mu)} \leq M e^{-\kappa t}$. Applying $\widetilde P(t)$ to the function $\overline f = f - \pi(f)$ yields the stated result.
\end{proof}

\section{The symmetric case}
\label{sec:symmetric}

Let us now consider Assumption~\ref{ass:symmetric} (Symmetry): $U(x) = U(-x)$ for all $x$. In this case the zigzag semigroup decouples into two separate semigroups:
\begin{itemize}
	\item[(i)] a `climb-fall' semigroup acting on functions $f(x,\theta)$ which satisfy $f(x,\theta) = f(-x,-\theta)$, and 
	\item[(ii)] a second semigroup acting on functions $f(x,\theta)$ which satisfy $f(x,\theta) = -f(-x,-\theta)$.
\end{itemize}
The climb-fall semigroup (i) can be interpreted as the semigroup corresponding to the process $Y(t)$, where $Y(t) = \Theta(t) X(t)$ with $(X(t),\Theta(t))$ denoting the Zig-Zag process, as described in the introduction. This process $Y(t)$ can (under Assumption~\ref{ass:symmetric}) be seen to be a (piecewise deterministic) Markov process itself, with generator~\eqref{eq:generator-climbfall} below. The other semigroup (ii) is not a Markov semigroup.

We will now discuss this decomposition in more detail.
Define an operator $T \in L(L^2(\mu))$ by 
\[ T f(x,\theta) = \frac 1 {\sqrt{2}}( f^+(x)  + \theta f^{-}(-x)).\]
Then $T$ is unitary: Indeed, if $f, g \in L^2(\mu)$, then
\begin{align*}
\langle T f, T g \rangle & = \int_{\R} (T f)^+(x) (T g)^+(x) \, e^{-U(x)} \, d x + \int_{\R} (T f)^-(x) (T g)^-(x) \, e^{-U(x)} \, d x \\
& = \frac 1 2 \int_{\R} (f^+(x) + f^-(-x))(g^+(x) + g^-(-x)) \, e^{-U(x)} \, d x  \\
& \quad \quad + \frac 1 2 \int_{\R} (f^+(x) - f^-(-x))(g^+(x) - g^-(-x)) \, e^{-U(x)} \, d x \\
& = \int_{\R} \left[ f^+(x) g^+(x) + f^-(-x) g^-(-x) \right] e^{-U(x)} \, d x = \langle f, g\rangle,
\end{align*}
where the last equality holds due to Assumption~\ref{ass:symmetric}.
It may be checked that
\[T^{\star} f(x,\theta)= T^{-1} f(x,\theta) = \frac 1 {\sqrt{2}} (f(\theta x,+1) + \theta f(\theta x,-1)).\]

\begin{remark}
	Define subspaces of $L^2(\mu)$ as follows:
	\begin{align*} \mathcal H_S & := \{ f \in L^2(\mu) : f^+(x) = f^-(-x) \ \text{for} \ x \in \R\}, \quad & \mathcal H^+ & := \{ f \in L^2(\mu) : f^-(x)= 0 \ \text{for} \ x \in \R\} \\
	\mathcal H_A & := \{ f \in L^2(\mu) : f^+(x) = - f^-(-x) \ \text{for} \ x \in \R\}, \quad & \mathcal H^- & := \{ f \in L^2(\mu) : f^+(x)= 0 \ \text{for} \ x \in \R\}.
	\end{align*}
	We can write $f \in L^2(\mu)$ as
	\[ \begin{pmatrix} f^+(x) \\ f^-(x) \end{pmatrix} = \frac 1 2 \begin{pmatrix} f^+(x) + f^-(-x) \\ f^-(x) + f^+(-x) \end{pmatrix} + \frac 1 2 \begin{pmatrix} f^+(x) - f^-(-x) \\ f^-(x) - f^+(-x) \end{pmatrix},\]
	coressponding with the decomposition $L^2(\mu) \cong \mathcal H_S \oplus \mathcal H_A$.
	We thus have that $L^2(\mu) \cong \mathcal H^+ \oplus \mathcal H^-$, $L^2(\mu) \cong \mathcal H_S \oplus \mathcal H_A$, $\mathcal H^+ \perp \mathcal H^-$ and moreover, using Assumption~\ref{ass:symmetric}, $\mathcal H_S \perp \mathcal H_A$.
	Furthermore 
	\[ T \mathcal H_S = \mathcal H^+, \quad T \mathcal H_A = \mathcal H^-, \quad T^{\star}\mathcal H^+ = \mathcal H_S, \quad T^{\star} \mathcal H^- = \mathcal H_A.\]
	In other words $T : \mathcal H_S \oplus \mathcal H_A \rightarrow \mathcal H^+ \oplus \mathcal H^-$ with the mapping $T$ respecting the direct sum.
	
	The motivation for this transformation and these subspaces stems from the following proposition.
\end{remark}

\begin{proposition}
	Suppose Assumptions~\ref{ass:sobolev} and~\ref{ass:symmetric} are satisfied, and $\lambda_{\refr}(x) = \lambda_{\refr}(-x)$ for all $x$.
	The transformation of $(L, \mathcal D(L))$ by $T$, i.e., $(\tilde L, \mathcal D(L))$ with $\widetilde L = T L T^{\star}$, is given by $\mathcal D(\widetilde L) = W^{1,2}(\mu)$ and 
	\[ (\widetilde L f)^{\pm}(x) = \partial_x f^{\pm}(x) + \lambda^+(x) (\pm f^{\pm}(-x) - f^{\pm}(x)), \quad f \in \mathcal D(\widetilde L).\]
\end{proposition}

This proposition states that the subspaces $\mathcal H^{\pm}$ are invariant for $\widetilde L$.
In terms of the original generator $L$, this implies that the decomposition $L^2(\mu) = \mathcal H_S \oplus \mathcal H_A$ is invariant under $L$. That is, $L$ maps $f \in \mathcal H_S \cap \mathcal D(L)$ into $\mathcal H_S$ and $f \in \mathcal H_A \cap \mathcal D(L)$ into $\mathcal H_A$ respectively.

\begin{proof}
	By Theorem~\ref{thm:compact-resolvent} we have $\mathcal D(L) = W^{1,2}(\mu)$  which is invariant under $T$ and its inverse.
	Write $g(\cdot,\theta) = T^{\star} f(\cdot,\theta) = \frac 1 {\sqrt 2} (f^+(\theta \cdot) + \theta f^-(\theta \cdot))$. Then 
	\[ \theta \partial_x g(\cdot,\theta) = \frac 1 {\sqrt 2} \left[(\partial_x f^+)(\theta \cdot) + \theta (\partial_x f^-)(\theta \cdot) \right].\]
	Using the above expression, and $\lambda^+(x) = \lambda^-(-x)$,
	\begin{align*}
	& (T L T^{\star} f)^+(x) \\
	& = \frac 1 {\sqrt 2} \left(L T^{\star} f^+(x) + L T^{\star} f^-(-x)\right) \\
	& = \frac 1 {\sqrt 2} \left(\partial_x g^+(x) - \partial_x g^-(-x) \right) \\
	& \quad \quad + \frac 1 {\sqrt 2} \left( \lambda^+(x) (g^-(x) - g^+(x)) + \lambda^-(-x) (g^+(-x) - g^-(-x)) \right) \\
	& = \frac 1 2 \left( \partial_x f^+(x) + \partial_x f^-(x) + \partial_x f^+(x) - \partial_x f^-(x) \right) \\
	& \quad \quad + \frac 1 2 \lambda^+(x) \left( f^+(-x) - f^-(-x) - f^+(x) - f^-(x) + f^+(-x) + f^-(-x) - f^+(x) + f^-(x) \right) \\
	& = \partial_x f^+(x) + \lambda^+(x) \left( f^+(-x) - f^+(x) \right).
	\end{align*}
	An analogous computation yields 
	\[(TLT^{\star} f)^-(x) = \partial_x f^-(x) - \lambda^+(x) \left(f^-(-x) + f^-(x) \right).\]
\end{proof}

Recall that $\nu$ denotes the measure on $\R$ with density $e^{-U}$ with respect to Lebesgue measure. We see that under the transformation $T$, the generator decouples, and we may thus consider the two generators $L^{\pm}$ on the decoupled spaces given by
\begin{equation} \label{eq:generator-climbfall} L^+ f(x) = f'(x) + \lambda^+(x) (f(-x) - f(x)), \quad f \in W^{1,2}(\nu), \end{equation}
and
\[ L^- f(x) = f'(x) + \lambda^+(x) (-f(-x) - f(x)), \quad f \in W^{1,2}(\nu)\]
independently.

\subsection{Spectral theory of the reduced semigroups}

Here we consider the two semigroups with generator
\[ L^{\pm} f(x) = f'(x) + \lambda^+(x) (\pm f(-x) - f(x)), \quad f \in W^{1,2}(\nu),\]
in $L^2(\nu)$. 
Because from an analytic point of view the two semigroups $L^+$ and $L^-$ are very similar, we carry out the analysis for the semigroups at once, at least as much as possible. We carefully keep track of the effect of the sign $\pm$ in the expressions that follow.
We assume throughout Assumption~\ref{ass:symmetric}, i.e., $U(x) = U(-x)$, $x \in \R$.

Define functions $Z^{\pm} : \C \rightarrow \C$ by 
\begin{equation}
\label{eq:Z-pm} Z^{\pm}(\gamma) =  1 \mp \int_0^{\infty} U'(\eta) e^{-2 \gamma \eta - U(\eta)} \, d \eta = 1 \mp \psi^+(\gamma)
\end{equation}
and sets
$\Sigma^+ \subset \C$ and $\Sigma^- \subset \C$ by
\begin{equation} \label{eq:Sigma} \Sigma^{\pm} = \left \{ \gamma \in \C :Z^{\pm}(\gamma) = 0 \right\}.\end{equation}

\begin{theorem} \label{thm:spectrum-symmetric} Suppose Assumptions~\ref{ass:sobolev}, \ref{ass:compactness}, \ref{ass:unimodal} and~\ref{ass:symmetric} are satisfied. Furthermore assume $\lambda_{\refr}(x) = 0$ for all $x$. Then $\sigma_p(L^{\pm}) = \sigma(L^{\pm}) = \Sigma^{\pm}$.	
	If $\gamma \in \rho(L^{\pm})$ then, for $h \in L^2(\nu)$, $f^{\pm} = (\gamma - L^{\pm})^{-1} h \in W^{1,2}(\nu)$ admits the expression
	\begin{equation} \label{eq:resolvent-symmetric}
	f^{\pm}(x) = \begin{cases}
	e^{\gamma x} \left( \frac{m^{\pm}_{\gamma}(h)}{Z^{\pm}(\gamma)} + \int_x^0 e^{-\gamma \xi} h(\xi) \, d \xi \right), \quad & x \leq 0, \\
	e^{\gamma x + U(x)} \int_x^{\infty} e^{-\gamma \xi - U(\xi)} \left[ h(\xi) \pm U'(\xi) e^{-\gamma \xi} \left(\frac{m^{\pm}_{\gamma}(h)}{Z^{\pm}(\gamma)} + \int_{-\xi}^0 e^{-\gamma \eta} h(\eta) \, d \eta \right)\right] \, d \xi, \quad & x > 0,
	\end{cases}
	\end{equation}
	where $m^{\pm}_{\gamma}(h)$ is given by
	\begin{equation} \label{eq:m-plusmin} m^{\pm}_{\gamma}(h) = \int_0^{\infty} e^{-\gamma \xi - U(\xi)} \left[ h(\xi) \pm U'(\xi) e^{-\gamma \xi}\int_{0}^{\xi} e^{\gamma \eta} h(-\eta) \, d \eta \right] \, d \xi.
	\end{equation}
	For every $\gamma \in \sigma(L^{\pm})$ the associated space of eigenfunctions is one-dimensional and spanned by $f_{\gamma}^{\pm} \in W^{1,2}(\nu)$ given by 
	\begin{equation}
	\label{eq:eigenfunctions-symmetric}
	f^{\pm}_{\gamma}(x) = \begin{cases} e^{\gamma x}, \quad & x \leq 0, \\
	\pm e^{\gamma x + U(x)} \int_x^{\infty} U'(\xi) e^{-2 \gamma \xi - U(\xi)} \, d \xi, \quad & x \geq 0. \end{cases}
	\end{equation}
\end{theorem}

In this proposition, and in the remainder of this manuscript, statements involving $\pm$ hold for the $+$ and $-$ cases, respectively. For example, $\sigma_p(L^+) = \sigma(L^+) = \Sigma^+$, and  $f^+_{\gamma}$ is an eigenfunction corresponding to the eigenvalue $\gamma \in \Sigma^+$ for the operator $L^+$.

\begin{proof}
	The details of this proof (e.g., bounds on integrals) are analogous to those given in the proof of Theorem~\ref{thm:spectrum} and will be omitted. We only carry out the computations yielding the stated expressions.
	First suppose $\gamma \in \rho(L^{\pm})$.
	The resolvent equation is $\gamma f^{\pm} - L^{\pm} f^{\pm} = h$ where $h \in L^2(\nu)$. For $x \leq 0$, since $\lambda^+(x) = 0$, this becomes $\gamma f^{\pm}(x) - (f^{\pm})'(x) = h(x)$, leading to the expression
	\begin{equation} \label{eq:f-negative-x} f^{\pm}(x) = e^{\gamma x} \left( k^{\pm}(\gamma;h) + \int_x^0 e^{-\gamma \xi} h(\xi) \, d \xi \right), \quad x \leq 0,\end{equation}
	where $k^{\pm}(\gamma;h)\in \C$ are integration constants depending on $\gamma$ and $h$, which will be specified below.
	For $x \geq 0$ the resolvent equation may be written as
	\begin{equation} \label{eq:resolvent-positive-x}\gamma f^{\pm}(x) - (f^{\pm})'(x) + U'(x) f^{\pm}(x) = \widetilde h^{\pm}(x), \quad x \geq 0,\end{equation}
	where, for $x \geq 0$,
	\begin{equation*} \label{eq:tilde-h} \widetilde h^{\pm}(x) := h(x) \pm U'(x) f^{\pm}(-x) = h(x) \pm U'(x) e^{-\gamma x} \left[  k^{\pm}(\gamma;h) + \int_{-x}^0 e^{-\gamma \xi} h(\xi)\, d \xi \right].\end{equation*}
	This first order ordinary differential equation admits the solution
	\begin{equation} \label{eq:f-positive-x-1} f^{\pm}(x) = e^{\gamma x +U(x)} \left( k^{\pm}(\gamma;h) - \int_0^x e^{-\gamma \xi - U(\xi)} \widetilde h^{\pm}(\xi) \, d \xi \right), \quad x \geq 0,\end{equation}
	where the integration constant is chosen so that $f^{\pm}$ is continuous in $x$.
	In case $\gamma \notin \Sigma^{\pm}$, let $k^{\pm}(\gamma;h) = m^{\pm}_{\gamma}(h)/Z^{\pm}(\gamma)$ with $m^{\pm}_{\gamma}(h)$ as in~\eqref{eq:m-plusmin}. Analogously to the proof of Theorem~\ref{thm:spectrum},
	it may then be verified that $f^{\pm}(x)$ admits the expression
	\begin{equation} \label{eq:f-positive-x-2} f^{\pm}(x) = e^{\gamma x+ U(x)} \int_x^{\infty} e^{-\gamma \xi - U(\xi)} \widetilde h^{\pm}(\xi) \, d \xi, \quad x \geq 0.\end{equation}
	Indeed, expression~\eqref{eq:f-positive-x-2} is equal to~\eqref{eq:f-positive-x-1} if and only if
	\begin{equation}\label{eq:condition-k} k^{\pm}(\gamma;h) = \int_0^{\infty} e^{-\gamma \xi - U(\xi)} \widetilde h^{\pm}(\xi) \, d \xi = \pm k^{\pm}(\gamma;h) \int_0^{\infty} U'(\xi) e^{-2 \gamma \xi -U(\xi)} \, d \xi + m_{\gamma}^{\pm}(h).\end{equation}
	Solving for $k^{\pm}(\gamma;h)$ and using~\eqref{eq:psi+-} yields the stated expression for $k^{\pm}(\gamma;h)$.
	We have, using~\eqref{eq:f-negative-x} and Lemma~\ref{lem:convolution} for $x \le 0$, and \eqref{eq:f-positive-x-2} and Lemma~\ref{lem:integrals} for $x \ge 0$, that $f^{\pm} \in L^2(\nu)$. This fully determines the solution $f^{\pm}$ of the resolvent equation.
    If $\gamma \in \Sigma^{\pm}$ then~\eqref{eq:condition-k} implies that $m_{\gamma}^{\pm}(h)$ which is defined by~\eqref{eq:m-plusmin} is identically equal to 0; a contradiction.
%
	
	Now suppose $\gamma \in \sigma(L^{\pm})$. Since $\sigma(L^{\pm}) \subset \sigma(L)$, it follows by Corollary~\ref{cor:compact-resolvent} that $\gamma$ is in the point spectrum of $L^{\pm}$. A solution to the equation $(\gamma - L^{\pm}) f^{\pm}_\gamma$ is given by
	\[ f^{\pm}_{\gamma}(x) = \begin{cases} e^{\gamma x}, \quad & x \leq 0, \\
	e^{\gamma x + U(x)} \left( 1 \mp \int_0^x U'(\xi) e^{-2 \gamma \xi - U(\xi)} \, d \xi \right), \quad & x \geq 0, \end{cases}
	\]
	insisting upon continuity in $x = 0$. Since we require $f^{\pm}_{\gamma} \in L^2(\nu)$, its growth as $x \rightarrow \infty$ should be controlled, requiring that
	\[ 1 \mp \int_0^{\infty} U'(\xi) e^{-2 \gamma \xi - U(\xi)} \, d \xi = 0,\]
	with the respective limits being well defined by an application of Lemma~\ref{lem:integrals}.
	But this is equivalent to $\gamma \in \Sigma^{\pm}$, and yields the alternative expression~\eqref{eq:eigenfunctions-symmetric}.
\end{proof}

\begin{remark}[Adjoint generator]
Introduce the operator $J : L^2(\nu) \rightarrow L^2(\nu)$ by $J g(x) = g(-x)$. It may be verified that $(L^{\pm})^{\star} = J L^{\pm} J$, a very similar situation as in the previous case where $L^{\star} = F LF$. 
Indeed by partial integration (taking for simplicity $f,g \in C_c^{\infty}(E)$ to be real valued),
\begin{align*}
\int_{\R} f L^{\pm} g  \, d \nu & = \int_{\R} f(x) \left[ g'(x) + \lambda^+(x) \left( \pm g(-x) - g(x) \right) \right] e^{-U(x)} \, d x \\
& = \int_{\R} \left[ - f'(x)  g(x) + U'(x) f(x) g(x) \pm \lambda^+(x) f(x) g(-x) - \lambda^+(x) f(x) g(x) \right] e^{-U(x)} \, d x \\
& = \int_{\R} \left[ -f'(x) g(x) \pm \lambda^+(-x) f(-x) g(x) - \lambda^-(x) f(x) g(x)  \right] e^{-U(x)} \, d x \\
& = \int_{\R} \left[ J \partial_x J f(x) \pm \lambda^+(-x) f(-x) - \lambda^+(-x) f(x) \right] g(x) e^{-U(x)} \, d x \\
& = \int_{\R} (JL^{\pm} J f) g \, d \nu,
\end{align*}
where we used partial integration in the second inequality, performed a change of variables $x \rightarrow -x$ for the integral over $\lambda^+(x) f(x) g(-x)$, and used $U'(x) - \lambda^+(x) = -\lambda^-(x) = -\lambda^+(-x)$ in the third and fourth inequality.

The observation that $(L^{\pm})^{\star} = J L^{\pm} J$ implies that the results of Lemma~\ref{lem:J-selfadjoint} concerning the spectrum of $L^{\pm}$ apply.
\end{remark}

\begin{corollary}[Spectral projection]
	\label{cor:spectral-projection-symmetric}
	Under the assumptions of Theorem~\ref{thm:spectrum-symmetric} the spectral projection corresponding to $\gamma \in \sigma(L^{\pm})$ is given by
	\[ P_{\gamma}^{\pm} h  = \frac 1 {2 \pi \i} \int_{\Gamma} k^{\pm}(\zeta, h) f_{\zeta}^{\pm}\, d \zeta,\]
	where $\Gamma$ is a Jordan contour in $\C$ enclosing the eigenvalue $\gamma$ and no other eigenvalues, where $k^{\pm}$ and $f_{\zeta}^{\pm}$ are as defined in Theorem~\ref{thm:spectrum-symmetric}.
	
	In particular if $\gamma$ is a simple root of $Z^{\pm}$, then $P_{\gamma}^{\pm}$ has rank one and is given by
	\[ P_{\gamma}^{\pm} h = \frac{m^{\pm}_{\gamma}(h)}{\frac{d}{d \gamma} Z^{\pm}(\gamma)} f_{\gamma}^{\pm} = \frac{\langle h, J \overline f^{\pm}_{\gamma} \rangle_{L^2(\nu)}}{\langle f^{\pm}_{\gamma}, J \overline f^{\pm}_{\gamma} \rangle_{L^2(\nu)}} f_{\gamma}^{\pm},\]
	with $m^{\pm}_{\gamma}$ given in Theorem~\ref{thm:spectrum-symmetric}.
\end{corollary}

\begin{proof}
	As in the proof of Proposition~\ref{cor:spectral-projection} the non-vanishing terms after integrating the resolvent along a simple closed contour $\Gamma$ enclosing the eigenvalue $\gamma$ are only those depending on $k^{\pm}(\gamma;h)$ since these have $Z^{\pm}(\gamma)$ in the denominator and are thus non-holomorphic on the interior of $\Gamma$. This yields the stated expression. The expressions for the case of a simple root of $Z^{\pm}$ follow analogously to the proof of Corollary~\ref{cor:simple-root}
\end{proof}

\begin{remark}
	As was the case with Remark~\ref{rem:expression-Z-prime}, it may be verified by direct computation that, if $\gamma \in \sigma_p(L^{\pm})$, we have
	\[ m_{\gamma}^{\pm}(h) = \langle h, J \overline f_{\gamma}^{\pm} \rangle \quad \text{and} \quad \frac{d}{d\gamma} Z^{\pm}(\gamma) = \langle f^{\pm}_{\gamma}, J \overline f^{\pm}_{\gamma} \rangle, \]
	also when $\frac{d}{d\gamma} Z^{\pm}(\gamma) = 0$.
\end{remark}

\section{Bounded perturbations}
\label{sec:perturbation}

So far the characterization of the spectrum relies upon the assumption that $\lambda_{\refr}(x) = 0$ for all $x$, in which case relatively explicit solutions to resolvent and eigenvalue equations can be obtained. In order to extend our analysis to the case $\lambda_{\refr} \neq 0$, we investigate the effect of small perturbations of the zigzag generator on its spectrum.

\begin{proposition} \label{prop:perturbation}
	Let $(L,\mathcal D(L))$ denote the generator of the zigzag semigroup in $L^2(\mu)$ and let $B$ be a bounded operator on $L^2(\mu)$. Suppose the assumptions of Theorem~\ref{thm:spectrum} are satisfied. Suppose $\gamma$ is a root of $Z$, so that $\gamma \in \sigma_p(L)$.
	\begin{itemize}
		\item[(i)] For $\epsilon > 0$ sufficiently small 
		there is a set of (repeated) eigenvalues $\gamma_j(\epsilon) \in \sigma_p(L + \epsilon B) = \sigma (L + \epsilon B)$, $j = 1, \dots, m$, with total algebraic multiplicity equal to the algebraic multiplicity $m$ of $\gamma$, such that $\gamma_j(\epsilon) \rightarrow \gamma$ as $\epsilon \downarrow 0$. Furthermore the eigenvalues $\gamma_j(\epsilon)$ admit the asymptotic expansions
		\[ \gamma_j(\epsilon) = \gamma + \epsilon \mu_j^{(1)} + o(\epsilon), \quad j =1,\dots,m,\]
		where $(\mu_j^{(1)})$ are the repeated eigenvalues of the operator $P_{\gamma} B P_{\gamma}$ considered in the $m$-dimensional space $P_{\gamma} L^2(\mu)$, where $P_{\gamma}$ denotes the spectral projection corresponding to $\gamma$.
		\item[(ii)] If moreover $\gamma$ is a simple root of $Z$, so that it is a simple eigenvalue of $L$, then for $\epsilon> 0$ sufficiently small the operator $L + \epsilon B$ has a simple eigenvalue $\gamma(\epsilon)$ such that $\gamma(\epsilon) \rightarrow \gamma$ as $\epsilon \downarrow 0$, and $\gamma(\epsilon)$ admits the asymptotic expansion
		\begin{equation} \label{eq:asymptotic-expansion} \gamma(\epsilon) = \gamma + \epsilon \langle B f_{\gamma}, F \overline f_{\gamma} \rangle_{L^2(\mu)} / \langle f_{\gamma}, F \overline f_{\gamma} \rangle_{L^2(\mu)} + o(\epsilon),\end{equation}
		where $f_{\gamma}$ is the eigenfunction associated to $\gamma$ as given by Theorem~\ref{thm:spectrum}.
	\end{itemize}
\end{proposition}

\begin{proof}
	Since $L$ has compact resolvent, also $L + \epsilon B$ has compact resolvent by \cite[Proposition III.1.12]{EngelNagel2000}, so that $\sigma(L + \epsilon B) = \sigma_p(L+ \epsilon B)$.
	Using that the dependence of $L + \epsilon B$ on $\epsilon$ is holomorphic, by \cite[Theorem VII.1.3]{Kato1995} the resolvent $(\zeta- L - \epsilon B)^{-1}$ is jointly holomorphic in $(\zeta, \epsilon)$ for sufficiently small $\epsilon$. Then \cite[Theorem VII.1.7]{Kato1995} establishes that spectral projections depend holomorphically on $\varepsilon$, in the sense of \cite[Equation (VII.1.4)]{Kato1995}. 
	It then follows that $\gamma$ is a stable eigenvalue in the terminology of~\cite[Section VIII.1.4]{Kato1995}): condition (i) of the definition of stable eigenvalue is satisfied by the mentioned joint analyticity of $(\zeta- L -\epsilon B)^{-1}$ around $\gamma$ and condition (ii) is satisfied by the mentioned analyticity of the spectral projections. We may therefore apply \cite[Theorem VIII.2.6]{Kato1995} to conclude (i). Finally (ii) follows from (i) and the expression for $P_{\gamma}$ obtained in Corollary~\ref{cor:simple-root}.
\end{proof}

\begin{remark}[Perturbations in the symmetric case] \label{rem:expansion-symmetric}
	If Assumption~\ref{ass:symmetric} is satisfied, and if additionally the bounded perturbation $B : L^2(\mu)$ respects the direct sum $B : \mathcal H_S \oplus \mathcal H_A \rightarrow \mathcal H_S \oplus \mathcal H_A$, then the perturbed operator $L + \epsilon B$ decomposes through the transformation $T$, i.e.,
	\[ T ( L + \epsilon B) T^{\star} = \begin{pmatrix}L^+  + \epsilon B^+ & 0 \\ 0 & L^- + \epsilon B^- \end{pmatrix}\]
	with $L^{\pm}$ as before and $B^{\pm} : L^2(\nu) \rightarrow L^2(\nu)$ bounded linear mappingss.
	If moreover the conditions of Theorem~\ref{thm:spectrum} are satisfied and $\gamma$ is a simple eigenvalue of $L^{\pm}$, then we obtain analogously to Proposition~\ref{prop:perturbation} an asymptotic expansion for the eigenvalues $\gamma(\epsilon)$ of the generators $L^{\pm} + \epsilon B^{\pm}$ given by
	\begin{equation} \label{eq:asymptotic-expansion-symmetric} \gamma(\epsilon) = \gamma + \epsilon \langle B^{\pm} f_{\gamma}, J \overline f_{\gamma}^{\pm} \rangle_{L^2(\nu)} / \langle f_{\gamma}^{\pm}, J \overline f_{\gamma}^{\pm} \rangle_{L^2(\nu)} + o(\epsilon),\end{equation}
	where $\gamma$ is a simple eigenvalue of $L^{\pm}$ and $f_{\gamma}^{\pm}$ is the associated eigenfunction as given by Theorem~\ref{thm:spectrum-symmetric} and we recall that $J f(x) = f(-x)$.
\end{remark}

\subsection{Example (additive perturbation of the switching intensity)}

Consider the bounded perturbation operator $B f = F f - f$, so that $L + \epsilon B$ admits the expression
\[ (L+\epsilon B) f(x,\theta) = \theta \partial_x f(x,\theta) + (\max( \theta U'(x), 0) + \epsilon ) (f(x,-\theta) - f(x,\theta)).\]
We see that the perturbation $\epsilon B$ effectively adds a constant $\epsilon$ to the switching intensity. 
In case the potential function $U$ is such that the assumptions of Theorem~\ref{thm:spectrum} are satisfied and furthermore all eigenvalues of $L$ are simple, then by Proposition~\ref{prop:perturbation}-(ii) the perturbed eigenvalues $\gamma(\epsilon)$ allow the asymptotic expansion
\begin{align*} \gamma(\epsilon) &= \gamma + \epsilon \langle B f_{\gamma}, F \overline f_{\gamma} \rangle_{L^2(\mu)} / \langle f_{\gamma}, F \overline f_{\gamma} \rangle_{L^2(\mu)} + o(\epsilon) \\
& = \gamma + \epsilon \left(\langle F f_{\gamma}, F \overline f_{\gamma} \rangle_{L^2(\mu)} / \langle f_{\gamma}, F \overline f_{\gamma} \rangle_{L^2(\mu)}  - 1 \right) + o(\epsilon) \\
& = \gamma + \epsilon \left(\langle f_{\gamma}, \overline f_{\gamma} \rangle_{L^2(\mu)} / \langle f_{\gamma}, F \overline f_{\gamma} \rangle_{L^2(\mu)}  - 1 \right) + o(\epsilon),
\end{align*}
where we subsequently used~\eqref{eq:asymptotic-expansion}, the expression for $B$, and finally the fact that $\langle F f, F g \rangle_{L^2(\mu)} = \langle f, g \rangle_{L^2(\mu)}$ for any $f, g \in L^2(\mu)$.

Furthermore the operator $B$ satisfies the condition stated in Remark~\ref{rem:expansion-symmetric} of preserving the decomposition $\mathcal H_S \oplus \mathcal H_A$. We compute 
\[ T B T^{\star} f(x,+1) = f(-x,+1) - f(x,+1), \quad \text{and} \quad TBT^{\star} f(x,-1) = - f(-x,-1) - f(x,-1),\]
so that in the notation of Remark~\ref{rem:expansion-symmetric}, we have 
\[ B^{\pm} f =  \pm  J f - f.\]
This means that if the additional assumption of symmetry (Assumption~\ref{ass:symmetric}, $U(x) = U(-x)$) is satisfied, then we may determine from~\eqref{eq:asymptotic-expansion-symmetric} the perturbed eigenvalues as
\begin{align*} \gamma(\epsilon) & = \gamma + \epsilon \langle B^{\pm} f_{\gamma}, J \overline f_{\gamma}^{\pm} \rangle_{L^2(\nu)} / \langle f_{\gamma}^{\pm}, J \overline f_{\gamma}^{\pm} \rangle_{L^2(\nu)} + o(\epsilon) \\
& =  \gamma + \epsilon \left( \langle \pm J f_{\gamma}^{\pm}, J \overline f_{\gamma}^{\pm} \rangle_{L^2(\nu)} / \langle f_{\gamma}^{\pm}, J \overline f_{\gamma}^{\pm} \rangle_{L^2(\nu)} - 1 \right) + o(\epsilon) \\
& = \gamma + \epsilon \left( \langle \pm f_{\gamma}^{\pm}, \overline f_{\gamma}^{\pm} \rangle_{L^2(\nu)} / \langle f_{\gamma}^{\pm}, J \overline f_{\gamma}^{\pm} \rangle_{L^2(\nu)} - 1 \right) + o(\epsilon), 
\end{align*}
where we should read `$+$' or `$-$' in the above expression depending on whether $\gamma \in \sigma(L^+) = \Sigma^+$ or $\gamma \in \sigma(L^-) = \Sigma^-$.

Unfortunately the expression does not seem to simplify further and in order to compute the above expansion in particular situations we will have to resort to numerical computation of the integrals involved.

\section{Examples}
\label{sec:examples}

\subsection{A family of medium-heavy to light-tailed symmetric distributions}
\label{sec:family}

Consider 
\begin{equation} \label{eq:symmetricfamily} U(x) = \frac 1 {\beta}\left[(1 + x^2)^{\beta/2} -1\right] \quad \text{where} \quad \beta > 1.\end{equation} 
We compute
\begin{align*}
U'(x) &=  x (1+x^2)^{\beta/2 -1}, \\
U''(x) & =  (1+x^2)^{\beta/2 -2} \left(1 +  (\beta - 1) x^2 \right).
\end{align*}

The case $\beta = 2$ corresponds to the standard normal distribution. Note that for $\beta \leq 1$ already the most basic assumptions of our theory (e.g., Assumption \ref{ass:compactness}) are not satisfied.

\begin{lemma}
	\label{lem:assumptions-family}
	Suppose $U$ is given by~\eqref{eq:symmetricfamily}. Then Assumptions~\ref{ass:sobolev}, \ref{ass:compactness}, \ref{ass:unimodal}, \ref{ass:polynomial-tail} and  \ref{ass:symmetric} are satisfied.
\end{lemma}

\begin{proof}
	Assumptions~\ref{ass:unimodal} and \ref{ass:symmetric} are immediate. From the expression for $U''(x)$ it follows that $U''(x) \geq 0$ for all $x \in \R$ and $\beta > 1$.
	We have
	\begin{align*}
	\frac{U''(x)}{(U'(x))^2} = (1 + x^2)^{-\beta/2} (1 + (\beta -1)x^2) x^{-2}  \sim (\beta-1) |x|^{-\beta},
	\end{align*}
	establishing that Assumption~\ref{ass:sobolev} is satisfied.
	Furthermore $|U'(x)| \sim |x|^{\beta -1} \rightarrow \infty$ as $|x| \rightarrow \infty$, establishing Assumption~\ref{ass:compactness}.
	
	Asymptotically, $U''(x) \sim (\beta - 1) |x|^{\beta - 2}$ as $|x| \rightarrow \infty$, which establishes Assumption~\ref{ass:polynomial-tail} for $\beta \geq 2$, using Lemma~\ref{lem:implications-m-convexity}.
%
	
	Finally we establish Assumption~\ref{ass:polynomial-tail} for $1 < \beta < 2$. Let $M > 1$. By symmetry it suffices to consider $y \geq x \geq M$ in Assumption~\ref{ass:polynomial-tail}. For all $x \geq M$, $(1+x^2) \leq \left(1 + \frac 1 {M^2} \right) x^2$. It follows that 
	\begin{equation} \label{eq:gradient-estimate} U'(x) = x (1+x^2)^{\beta/2-1} \geq \left( 1 + \frac 1 {M^2} \right)^{\beta/2 - 1} x^{\beta-1}, \quad x \geq M.\end{equation}
	Define $h(\xi) = U(x+\xi) - U(x) - m \xi^{\beta}$ for $\xi \geq 0$ and fixed $x \geq M$. Then, using~\eqref{eq:gradient-estimate}
	\begin{align*}
	h'(\xi) & = U'(x + \xi) - m \beta \xi^{\beta-1} \geq \left(1 + \frac 1 {M^2} \right)^{\beta/2 -1} \left( x + \xi \right)^{\beta - 1} - m \beta \xi^{\beta -1} \\
	&\geq \left[ \left(1 + \frac 1 {M^2} \right)^{\beta/2 -1}  - m \beta  \right] \xi^{\beta - 1} = 0,
	\end{align*}
	by taking $m = \frac 1 {\beta}  \left(1 + \frac 1 {M^2} \right)^{\beta/2 -1}$.
	We see that $h(\xi)$ is strictly non-decreasing as a function of $\xi$ for any $x \geq M$.
	Thus $h(y-x) \geq h(0) = 0$ for all $y$, establishing that
	\[ U(y) \geq U(x) + m |y-x|^{\beta}, \quad y \geq x \geq M,\]
	as required.
\end{proof}

\subsection{Gaussian distribution}
\label{sec:gaussian}

Suppose $U(x) = x^2/(2\sigma^2)$. Without loss of generality we assume that $\sigma = 1$ and consider $U(x) = x^2/2$. 
This is an instance of the family discussed above, with $\beta = 2$. In particular Assumptions~\ref{ass:sobolev} until~\ref{ass:symmetric} are satisfied.
Let $\erfc(z)= 1- \erf(z)$ denote the complimentary error function, with the error function $\erf(z)$ given by 
\[ \erf(z) = \frac 2 {\sqrt{\pi}} \int_0^z e^{-t^2} \, d t, \quad z \in \C.\]

\begin{proposition}
	Suppose $U(x) = x^2/2$ and $\lambda_{\refr}(x)=0$ for all $x$. Then the spectrum of the zigzag semigroup is given by 
	\[ \sigma(L) =\sigma_p(L) = \Sigma^+ \cup \Sigma^-,\]
	where
	\[ \Sigma^+ = \{ 0 \} \cup \{ \gamma \in \C:  \erfc(\sqrt{2} \gamma) = 0\} \quad \text{and} \quad 
	\Sigma^-  =  \{ \gamma \in \C : \sqrt{\pi}  \gamma e^{2 \gamma^2} \erfc(\sqrt{2} \gamma) = \sqrt{2}\}.\]
	All eigenvalues $\gamma \in \sigma(L)$ are algebraically simple.
\end{proposition}

\begin{proof}
	As established in Theorem~\ref{thm:spectrum-symmetric}, the spectrum consists of eigenvalues only, and is given by $\sigma(L) = \Sigma^+ \cup \Sigma^-$, where
	\[ \Sigma^{\pm} = \sigma(L^{\pm}) = \{ \gamma \in \C: \psi^+(\gamma) = \pm 1\}.\]
	By partial integration (Lemma~\ref{lem:psi})
	\begin{equation} \label{eq:psi-gaussian-case} \psi^+(\gamma) = 1 - \sqrt{2 \pi} \gamma e^{2 \gamma^2} \erfc(\sqrt{2} \gamma), \quad \gamma \in \C,\end{equation}
	resulting in the characterization of $\sigma^{\pm}$.
	From~\eqref{eq:psi_alternative} and~\eqref{eq:dpsi} we have
	\begin{equation} \label{eq:relation-psi-gaussian-case} \psi^+(\gamma) = 1 + \gamma \frac{d}{d\gamma} \psi^+(\gamma) - 4 \gamma^2 \psi^+ (\gamma). \end{equation}
	Note that $\psi^-(\gamma) = \psi^+(\gamma)$ by symmetry of $U$.
	Since $Z(\gamma) = 1 - (\psi^+(\gamma))^2$, it follows that $\frac{d}{d \gamma} Z(\gamma) = - 2 \psi^+(\gamma) \frac{d}{d\gamma} \psi^+(\gamma)$.
	Thus if $\gamma$ is an eigenvalue with algebraic multiplicity larger than one, then $\psi^{+}(\gamma) = \pm 1$ and $\frac{d \psi^+}{d \gamma} = 0$. Assume $\frac{d}{d \gamma}\psi^+(\gamma) = 0$. 
	Using~\eqref{eq:relation-psi-gaussian-case}, in case $\psi^+(\gamma) = 1$, then necessarily $\gamma = 0$; however it was established in the proof of Proposition~\ref{prop:dominant-eigenvalue} that $\frac{d}{d\gamma} Z(0) \neq 0$. In case $\psi^+(\gamma) = -1$, using again~\eqref{eq:relation-psi-gaussian-case}, we find that $\gamma = \pm \frac 1 2 \sqrt{2} \i$, which is however not a root of $\psi^+$ as defined in~\eqref{eq:psi-gaussian-case}.
\end{proof}

Asymptotic expansions for the elements of $\Sigma^+$ can be found in \cite{fettis1973complex}. However for a complete computation of the spectrum we have to rely on numerical computation.

\subsection{Numerical results}
\label{sec:numerics}

By Theorems~\ref{thm:spectrum} and~\ref{thm:spectrum-symmetric}, the spectrum of $L$ and $L^{\pm}$ may be computed as the roots of $Z$ (see \eqref{eq:Z}) and $Z^{\pm}$ (see \eqref{eq:Z-pm}), respectively. For ease of notation we only consider roots of  $Z$ but the exposition applies equally to roots of $Z^{\pm}$.

A suitable numerical routine for determining roots of $Z$ is provided by \cite{Dellnitz2002}, to which we refer for details. For our purposes it is important to remark that the method of \cite{Dellnitz2002} relies on the ability to integrate $Z'(\zeta)/Z(\zeta)$ efficiently along line segments in the complex plane. We have
\[ \frac{Z'(\zeta)}{Z(\zeta)} = - \frac{\psi^-(\zeta) \frac{d}{d\zeta} \psi^+(\zeta) + \psi^+(\zeta) \frac{d}{d\zeta} \psi^-(\zeta)}{1 - \psi^+(\zeta) \psi^-(\zeta)}.\]
In the case of a Gaussian target distribution (Section~\ref{sec:gaussian}), these functions may be expressed in terms of the error function. However for the general family considered in Section~\ref{sec:family} we have to rely on numerical integration to determine the value of $Z'(\zeta)/Z(\zeta)$.

Computer code (in Julia) for visualization of the spectrum and perturbations may be found at \cite{Bierkens2019githubspectral}. Graphical depictions of the spectrum, including the effect of perturbations by a positive switching rate, for the Gaussian case ($\beta =2$) and two other cases are provided in Figure~\ref{fig:plot}. 
In the Gaussian case the rightmost (pairs of) eigenvalues are given (up to 6 significant digits) by $0$, $-0.425665 \pm  1.02295\i$, $-0.957995 \pm 1.40818 \i  $, $-1.26616 \pm 1.66757\i , -1.53940 \pm  -1.90293\i$ and we find that the spectral gap is given by $0.425665$. An interesting observation is that a positive refreshment rate increases the spectral gap (since the spectrum moves towards the left). This phenomenon is also observed in \cite{MicloMonmarche2013}. Furthermore it seems that making the tail more heavy (i.e. decreasing $\beta$) results in an increase of the spectral gap.

\begin{figure}[ht!]
	{\begin{center}
			\begin{subfigure}{0.9 \textwidth}
				\includegraphics[width=\textwidth]{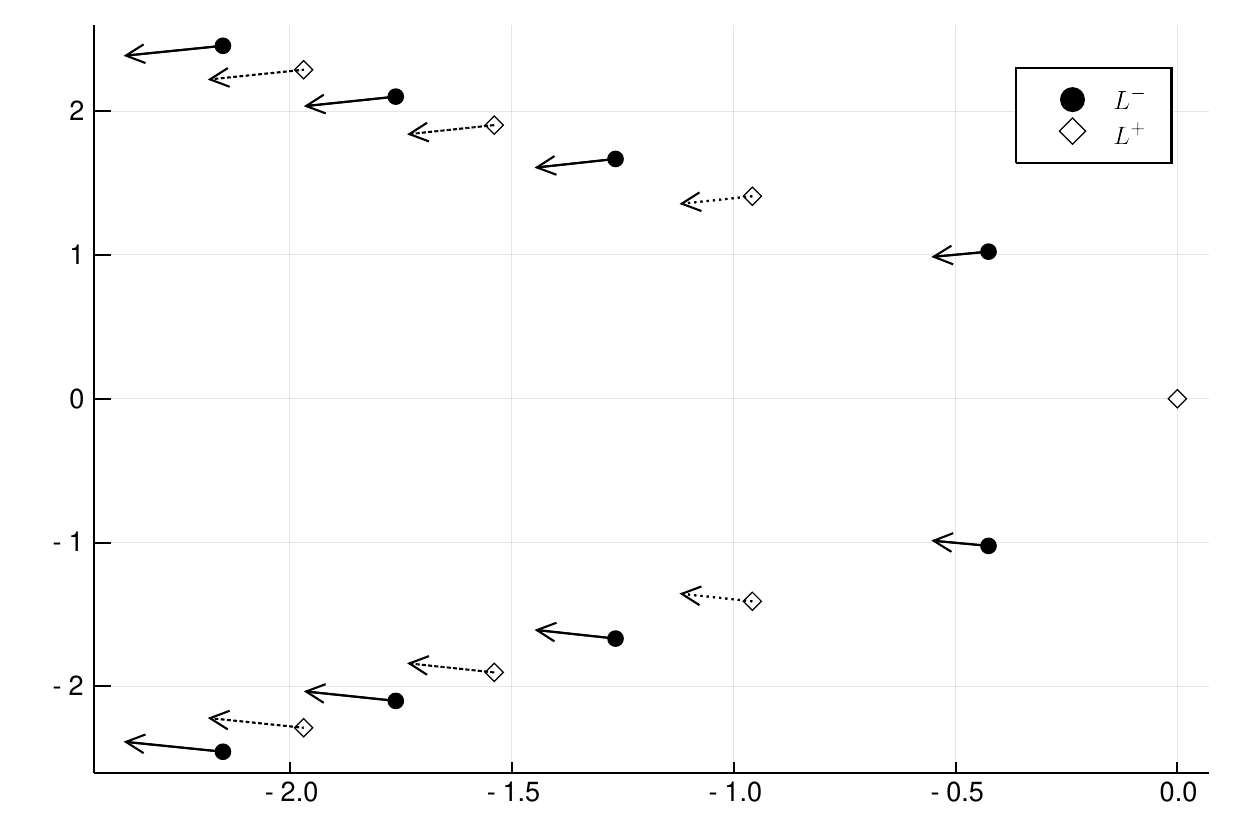}
				\caption{Standard normal distribution ($\beta = 2$).}
			\end{subfigure} \\
			\begin{subfigure}{0.45 \textwidth}
				\includegraphics[width=\textwidth]{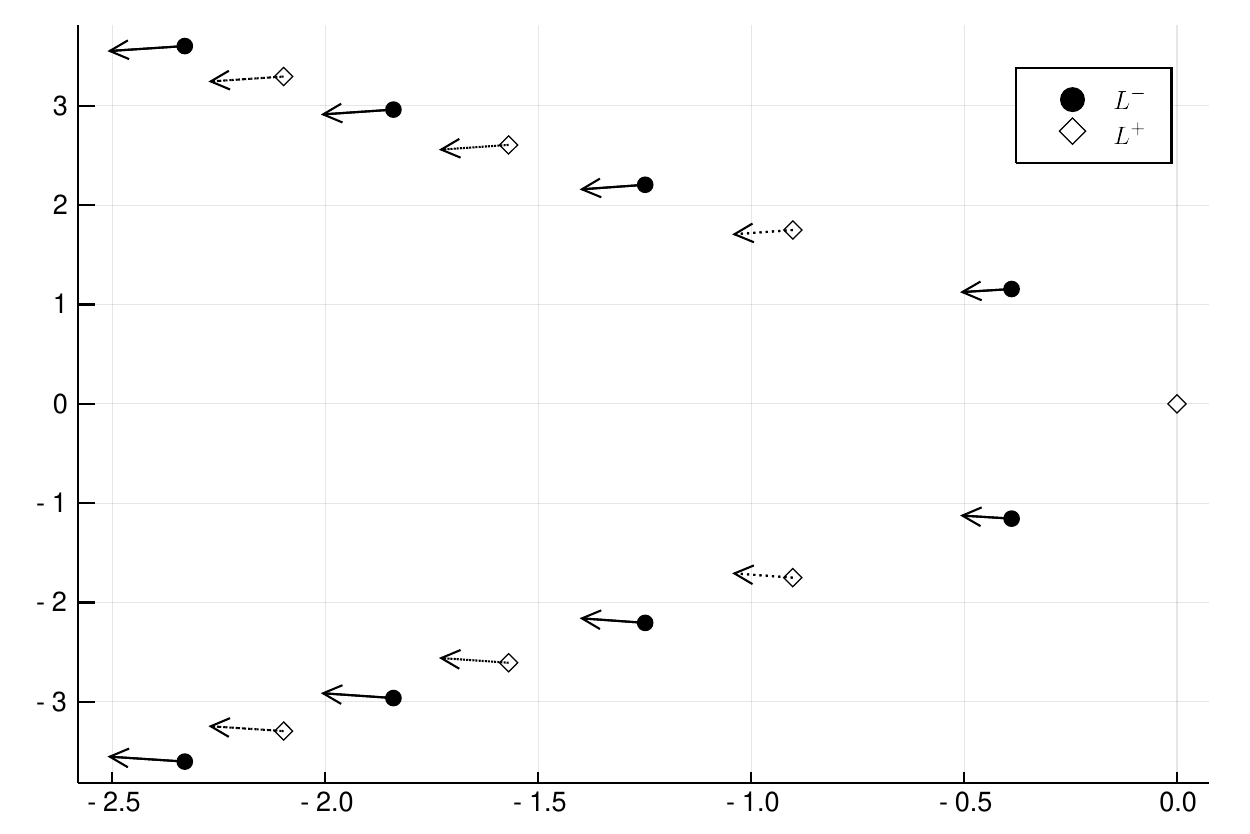}
				\caption{`Light tailed' case $\beta = 2.5$.}
			\end{subfigure}
			\begin{subfigure}{0.45 \textwidth}
				\includegraphics[width=\textwidth]{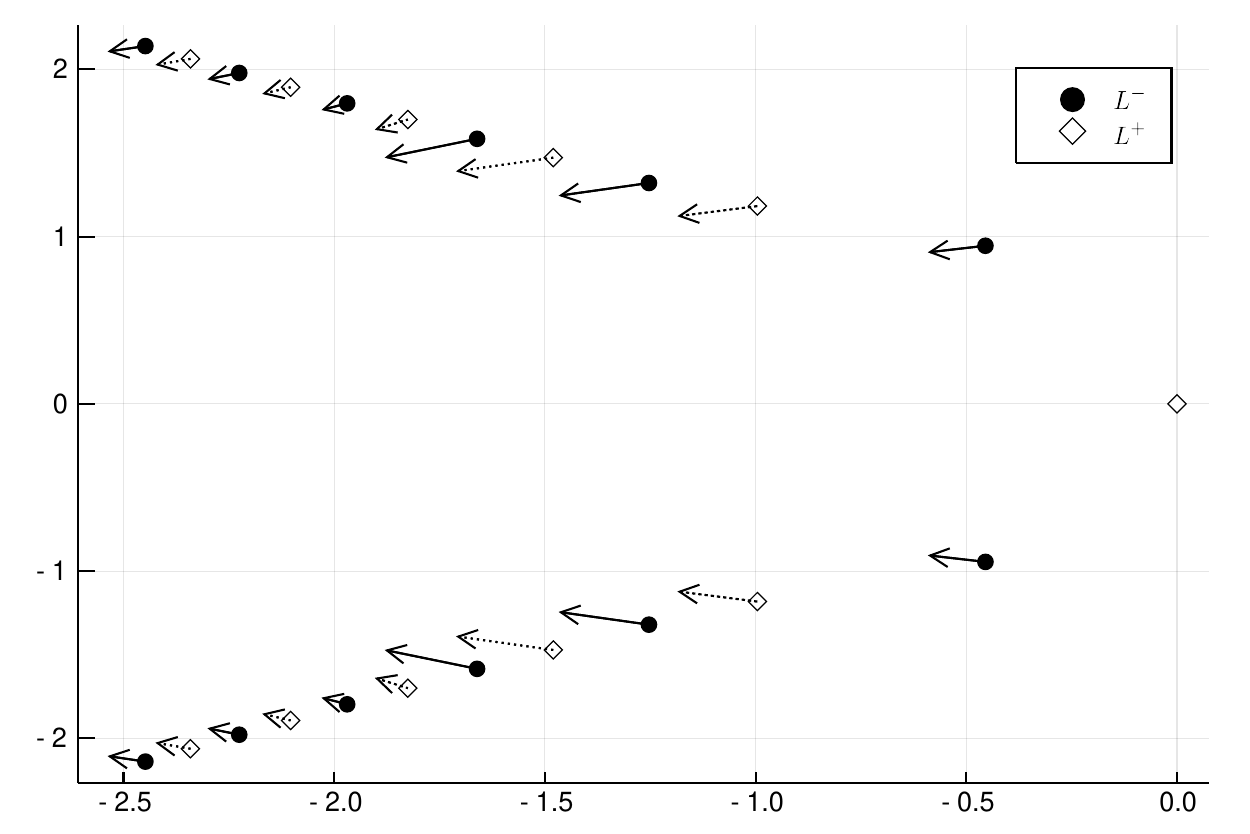}
				\caption{`Heavy tailed' case $\beta = 1.75$.}
			\end{subfigure}
			\caption{The spectra of $L^+$ and $L^-$ for the distributions described in Section~\ref{sec:family}, along with the directions in which the spectrum is perturbed under a small refreshment rate $\lambda_{\refr} = \epsilon > 0$.}
			\label{fig:plot}
	\end{center}}
\end{figure}

\section{Discussion}

In this section we will explore the relation between the results obtained in this paper, and other results concerning exponential convergence to equilibrium.
Assume $\mu(E) < \infty$ and 
recall the notation $\pi$ for the normalization of $\mu$, defined by~\eqref{eq:def-pi}.

\subsection{Connection between spectral gap and geometric ergodicity}

Suppose we have a spectral gap for the zigzag semigroup; e.g., assume the conditions of Theorem~\ref{thm:decay} are satisfied. We consider a time discretization of the zigzag semigroup $(P(t))_{t \geq 0}$ by defining, for a fixed $t_0 > 0$, the Markov kernel
\[ Q(\eta, A) := \int P(t_0)\1_A(x,\theta) \, \eta(dx, d \theta), \]
where $\eta$ is any probability measure on $E = \R \times \{-1,+1\}$.

For any signed measure $\eta \ll \pi$ we define $ \| \eta \|_{L^2(\pi)}^2 = \int \left(\frac{d \eta}{d \pi}\right)^2 \, d \pi$.
We say that a Markov chain with transition kernel $Q$ on a general state space $(X,\mathcal X)$ is \emph{$L^2(\pi)$-geometrically ergodic} if there is a function $c(\eta) <\infty$ and a constant $\rho < 1$ such that 
\[ \| Q^n(\eta, \cdot) - \pi \|_{L^2(\pi)} \leq c(\eta) \rho^n,  \]
for all Borel probability distributions $\eta$ on $E$ satisfying $\|\eta\|_{L^2(\pi)} < \infty$.

By considering the adjoint semigroup of the zigzag semigroup we then have the following result.

\begin{corollary}
Under the assumptions of Theorem~\ref{thm:decay}, the Markov kernel $Q$ is $L^2(\pi)$-geometrically ergodic.
\end{corollary}
\begin{proof}
We have that $Q^n(\eta,A) = \int_A P^{\star}(n t_0) \frac{d \eta}{d \pi}\, d \pi$ where $P^{\star}(t)$ is the adjoint semigroup of $P(t)$. Since the result of Theorem~\ref{thm:decay} for $(P(t))_{t \geq 0}$ carries over to $(P^{\star}(t))_{t \geq 0}$ (using Theorem~\ref{thm:semigroup}), the stated claim follows.
\end{proof}

An alternative notion of geometric ergodicity is the following. We say that a chain is \emph{$\pi$-almost everywhere geometrically ergodic} if there is a function $W : E \rightarrow \infty$ and $\rho < 1$ such that for $\pi$-almost all $z$,
\[ \| Q^n(z,\cdot) - \pi\|_{\mathrm{TV}} \leq W(z) \rho^n.\]

Importantly the latter  notion is phrased in terms of $\pi$-almost all initial positions. However, as a result of \cite[Theorem 1]{roberts2001geometric} we have indeed that $Q$ is also $\pi$-almost everywhere geometrically ergodic.
%
%
%
%

\subsection{Comparison to hypocoercivity analysis}

In \cite{Andrieu2018} conditions for hypocoercivity of piecewise deterministic process such as the zigzag process are established. We will rephrase the result of \cite{Andrieu2018} to fit in the present context.

\begin{theorem}[{\cite[Theorem 1]{Andrieu2018}}]
	\label{thm:hypocoercivity}
	Suppose the following assumptions hold:
	\begin{itemize}
		\item[(i)] $U\in C^3(\R)$, $\lambda_{\refr}$ is continuous and $e^{-U} \in L^1(\R)$.
		\item[(ii)] There is a $p \geq 0$ such that $\sup_{x \in \R} |U^{(3)}(x)| / (1 + |x|^p) < \infty$.
		\item[(iii)] There is a constant $c_1 \geq 0$ such that $U''(x) \geq - c_1$ for all $x \in \R$.
		\item[(iv)] $\liminf_{|x| \rightarrow \infty}\{ |U'(x)|^2/2 - U''(x)\} > 0$.
		\item[(v)] The refreshment rate $\lambda_{\refr} : \R \rightarrow (0,\infty)$ is bounded from below and from above as follows. There exists $\underline \lambda> 0$ and $c_{\lambda} \geq 0$ such that for all $x \in \R$,
		\[ 0 < \underline \lambda \leq \lambda_{\refr}(x) \leq \underline \lambda ( 1 + c_{\lambda} |U'(x)|).\]
	\end{itemize}
	Then there exist constants $C > 0$ and $\alpha > 0$ such that, for any $f \in L^2_0(\mu)$ and $t \geq 0$,
	\[ \|P(t) f\|_{L^2_0(\mu)} \leq C e^{-\alpha t} \| f\|_{L^2_0(\mu)}.\]
\end{theorem}
\begin{proof}
We may write the generator of the zigzag semigroup in the notation of \cite[Equation (1)]{Andrieu2018}, 
\[ L f(x,\theta) = \theta \partial_x f(x,\theta) + \lambda_1(x,\theta) ( \mathcal B_1 - I) f(x,\theta) + m_2^{1/2} \lambda_{\text{ref}}(x) \mathcal R_v f(x,\theta),\]
where
\begin{align*} \lambda_1 f(x,\theta) & = (\theta F_1(x))_+, \quad \text{for} \quad   F_1(x) = U'(x), \\
  \mathcal B_1 f(x,\theta) & = f(x,-\theta), \quad  m_2 = 1, \quad \lambda_{\text{ref}}(x) = 2 \lambda_{\refr}(x), \quad \mathcal R_v f(x,\theta) =\tfrac 1 2 (f(x,-\theta) - f(x,\theta)).
\end{align*}
(The factor 2 had to be redistributed between the refreshment rate and the refreshment kernel to be consistent with the definition of $\mathcal R_v$ in \cite{Andrieu2018}.)

We verify that the stated assumptions imply the assumptions A1 and H1-H6 of \cite{Andrieu2018}. Some notations below are from \cite{Andrieu2018}, and may override notation from this paper; in particular $\nu$ for the uniform marginal stationary distribution on the velocities $\{-1,+1\}$.

By Lemma~\ref{lem:core} and Theorem~\ref{thm:semigroup}, Assumption A1 of \cite{Andrieu2018} is satisfied.

By (ii), $U \in C^{(3)}_{\text{poly}}(\R)$. Together with assumptions (iii) and (iv) this implies H1.

We have $F_1(x) = U'(x)$, which by (i) is in $C^2(\R)$. It is then clear that H2 is satisfied.

H3 is satisfied by taking $\varphi(s) = \max(0,s)$.

H4 is satisfied by taking $V = \{-1,+1\}$ and $\nu =\tfrac 1 2 (\delta_{-1} + \delta_{+1})$.

H5 is satisfied for $\mathcal R_v :=  \Pi_v-I$ where $\Pi_v f = \nu(f)$. In particular
if $g \in L^2_0(\nu)$ we have, writing $g = (g_1, g_2)  =(g_1,-g_1)$,
\[ \langle -\mathcal R_v g, g \rangle_{L^2(\nu)} = - \tfrac 1 2 g_1^2 -\tfrac 1 2  g_2^2 + \tfrac 1 4 (g_1+g_2)^2 
= -\tfrac 1 4 (g_1-g_2)^2 = \|g\|_{L^2(\nu)} .\]

Finally H6 is identical to (v).

\end{proof}

Bounds for the constants $C$ and $\alpha$ can, in principle, be obtained in explicit form although this requires a tedious effort.
The result above may be compared to the same result of Theorem~\ref{thm:decay}, which depends on Assumptions~\ref{ass:sobolev}, \ref{ass:compactness}, \ref{ass:unimodal} and \ref{ass:polynomial-tail}. The above conditions are different from our conditions in several respects. Condition (iv) in the above theorem is stronger than our Assumption~\ref{ass:sobolev}, for example. Furthermore a strictly positive refreshment rate is required. On the other hand we require that $|U'(x)| \rightarrow \infty$ (Assumption~\ref{ass:compactness}), a zero refreshment rate and unimodal target distribution to obtain the explicit characterization of the eigenvalues.

\subsection{Comparison to exponential ergodicity}

In \cite{BierkensRoberts2017} the following result on exponential ergodicity of the zigzag process is established.  A function $V : E \rightarrow \R$ is called \emph{norm-like} if $V(x,\theta) > 0$ for all $(x,\theta) \in E$ and $\lim_{|x|\rightarrow \infty} V(x,\theta) = \infty$ for $\theta = \pm 1$.

\begin{theorem}[{\cite[Theorem 5]{BierkensRoberts2017}}]
	\label{thm:exponentialergodicity}
	Suppose $U$ is continuously differentiable, $\lambda_{\refr}$ is continuous, and for $\lambda^{\pm}(x) = (\pm U'(x) \vee 0) + \lambda_\refr(x)$, we have that 
	\[\liminf_{x \rightarrow \infty} \lambda^+(x) > \limsup_{x \rightarrow \infty}\lambda^-(x) \quad \text{and} \quad \liminf_{x \rightarrow -\infty} \lambda^-(x) > \limsup_{x \rightarrow -\infty} \lambda^+(x).\]
	Then there are constants $c > 0$, $\alpha > 0$ and a continuous norm-like function $V \in L^2(\mu)$, $V > 0$ on $E$, such that for all $(x,\theta) \in E$,
	\[ |P(t) f(x,\theta) - \pi(f)| \leq c (1 + V(x,\theta)) e^{-\alpha t}, \quad t \geq 0, \] for all measurable $f : E \rightarrow \R$ satisfying $|f(x,\theta)| \leq 1 + V(x,\theta)$ on $E$.
\end{theorem}

Note that here the family of operators $P(t)$ represents the extension of the Markov semigroup to a suitable space of measurable functions.
In the case with constant refreshment rate the condition of Theorem~\ref{thm:exponentialergodicity} reduces to the requirement that $\liminf_{x \rightarrow \infty} U'(x) > 0$ and $\limsup_{x \rightarrow -\infty} U'(x) < 0$.
We see that the above result implies, under only mild conditions, exponentially fast convergence in $L^2(\pi)$ for functions $f$ satisfying $|f(x,\theta)| \leq 1 + V(x,\theta)$. However, this class of functions does not contain the full space $L^2(\pi)$. Indeed, by the proof of \cite[Theorem 5]{BierkensRoberts2017}, the function $V$ is growing at an exponential rate as $|x| \rightarrow \infty$, putting a restriction on the growth of the functions $f(x,\theta)$. 

\subsection*{Acknowledgements}
We acknowledge helpful discussions with Sonja Cox, Mark Veraar, Pierre Monmarch\'e, Nils Berglund, Andi Wang and Boris Nectoux. We also acknowledge two reviewers and the associate editor for their valuable remarks.

%
%

\appendix

\section{Technical lemmas}

\subsection{Weak derivatives}

In this section we recall some basic results on weak derivatives as used in Section~\ref{sec:semigroup}.
A function $f : \R \rightarrow \R$ is said to have a weak derivative $f'$ if, for all $\varphi \in C_c^{\infty}(\R)$, we have
\[ \int_{\R} f \varphi' \, d x = - \int_{\R} f' \varphi\, d x.\]
If furthermore, for $p \geq 1$, $\left. f \right|_U$ and $\left. f' \right|_U \in L^p(U)$ for any open $U \subset \R$ with compact closure, then we write $f \in W^{1,p}_{\loc}(\R)$.

\begin{lemma}
\label{lem:productrule}
Suppose $f \in W^{1,p}_{\loc}(\R)$ and $g \in C^1(\R)$. Then $f g \in W^{1,p}_{\loc}(\R)$, and $(f g)' = f' g + fg'$.
\end{lemma}

\begin{proof}
Let $\varphi \in C_c^{\infty}(\R)$ and write $K = \operatorname{supp} \varphi$. By standard approximation arguments we may find $\tilde g \in C_c^{\infty}(R)$ such that $\sup_{x \in K} |\tilde g(x) - g(x)| + |\tilde g'(x) - g'(x)| < \varepsilon$. We then have by the definition of the weak derivative that
\[ \int_{\R} f \tilde g \varphi' \, d x = \int_{\R} f (\tilde g \varphi)' \, d x - \int_{\R} f \tilde g' \varphi \, d x = - \int_{\R} f' \tilde g \varphi \, d x - \int_{\R} f \tilde g' \varphi \, d x. \]
Since $\varepsilon > 0$ is arbitrary, and $\varphi$ has compact support, this yields
\[ \int_{\R} f g \varphi' \, d x =  - \int_{\R} ( f' g + f g') \varphi \, d x\]
so that $f g$ has weak derivative $f' g + f g'$. Now since $g$ and $g'$ are bounded on any bounded set, and $f \in W^{1,p}_{\loc}(\R)$, it follows that $f g \in W^{1,p}_{\loc}(\R)$.
\end{proof}

\subsection{Assumption~\ref{ass:polynomial-tail} revisited}

The following lemma establishes that the constant $M$ in Assumption~\ref{ass:polynomial-tail} can, without loss of generality, be assumed to be equal to zero.

\begin{lemma}
	\label{lem:polynomial-tail}
	Suppose Assumption~\ref{ass:polynomial-tail} is satisfied for certain values $C \leq 0$, $M >0$, $m > 0$ and $p > 1$. Then it is also satisfied for $M = 0$. That is, for $\widetilde m := 2^{1-p} m$ and some $\widetilde C \leq C$ we have that 
	\begin{equation} \label{eq:polynomial-tail-2} U(y) \geq U(x) + \widetilde C + \widetilde m|x-y|^p \quad \text{for all $x,y$ for which $y \geq x \geq 0$ or $y \leq x \leq 0$}.\end{equation}
\end{lemma}
\begin{proof}
	Suppose Assumption~\ref{ass:polynomial-tail} holds. We prove~\eqref{eq:polynomial-tail-2} for $y \geq x \geq 0$; the proof for $y \leq x \leq 0$ is analogous. First suppose $0 \leq x \leq y \leq M$. Write $C_1 := \max_{0 \leq \xi \leq M} U(\xi)$ and $C_2 := \min_{0 \leq \xi \leq M} U(\xi)$. Define $\widetilde C_1 = C_2 - C_1 - m M^p \leq 0$. Then 
	\[ U(y) \geq U(x) + \widetilde C_1 +  m|x-y|^p.\]
	
	Next consider the situation $0 \leq x \leq M \leq y$. We have by Assumption~\ref{ass:polynomial-tail} that 
	\[ U(y) \geq U(M) + C + m | M -y|^p.\]
	Next
	\[ U(M) \geq U(x) + \widetilde C_1 + m |M-x|^p\] by the first part of this proof. Finally we have by Jensen's inequality that
	\[ |y-M|^p + |M-x|^p \geq 2^{1-p} |y-x|^p.\]
	It follows that 
	\[ U(y) \geq U(x) + \widetilde C_1 + m |M-x|^p + C + m | M-y|^p \geq  U(x) + \widetilde C_1  + C + 2^{1-p} m |y-x|^p. \]
	By assumption the required inequality also holds for $M \leq x \leq y$.
	The statement of the lemma follows by taking $\widetilde C := \min(C, \widetilde C_1, C+ \widetilde C_1) = C + \widetilde C_1$.
\end{proof}

\subsection{The operator $C$ and its core.}

As in Section~\ref{sec:compactness}, for $f \in W^{1,2}_{\text{loc}}(\R)$ let $C f = f' - U f'$, and (with slight abuse of notation) restrict $C$ to the domain $\mathcal D(C) = \{ f \in L^2(\nu) \cap W^{1,2}_{\text{loc}}(\R) : C f \in L^2(\nu)\}$. 
\begin{lemma}
\label{lem:C-core}
 $(C, \mathcal D(C))$ is closed and $C_c^{\infty}(\R)$ is a core for $(C, \mathcal D(C))$.
\end{lemma}
\begin{proof}
First we show that $(C, \mathcal D(C))$ is closed. Let $(f_n) \subset \mathcal D(C)$, $g_n = C f_n$, and assume $f_n \rightarrow f$ and $g_n \rightarrow g$ in $L^2(\nu)$. 
Then $f_{n_k} \rightarrow  f$ and $g_{n_k} \rightarrow g$ $\nu$-almost everywhere for a suitable subsequence. Therefore
\[ f_{n_k}'(x) = g_{n_k}(x) + U'(x) f_{n_k}(x) \rightarrow g(x) + U'(x) f(x), \quad \text{$\nu$-almost everywhere}.\]
Now fix an open set $F \subset \R$ with compact closure. Since $U'$ is bounded on $F$, it follows that the subsequence $(f_{n_k}') = (g_{n_k} + U' f_{n_k})$ is Cauchy in $L^2(F)$. Thus $f \in W^{1,2}_{\loc}(\R)$. Next
\begin{align*}
 \int_{\R} |C f -g|^2 d \nu & = \int_{\R} \lim_{k \rightarrow \infty} | C f_{n_k} - g_{n_k}|^2 d \nu \stackrel{\text{Fatou}}{\leq} \lim_{k \rightarrow \infty} \int_{\R}  | C f_{n_k} - g_{n_k}|^2 d \nu = 0.
\end{align*}
Therefore $\|C f \|_{L^2(\nu)} \leq \|C f - g \|_{L^2(\nu)} + \|g\|_{L^2(\nu)} < \infty$ and $C f = g$.

Next we will show that $C_c^{\infty}(\R)$ is a core for $(C, \mathcal D(C))$, in complete analogy to Lemma~\ref{lem:core}.
Let $f \in \mathcal D(C)$. First take functions $\xi_n$ as in the proof of Lemma~\ref{lem:core} and define $f_n(x) := f(x) \xi_n(x)$. Then $f_n \in \mathcal D(C)$, $f_n$ has compact support and $f_n \rightarrow f$ in $L^2(\nu)$ by dominated convergence, and
\[ \|C f_n - C f\|_{L^2(\nu)} \leq \|(C f)(\xi_n - 1)\|_{L^2(\nu)} + \| f \xi_n'\|_{L^2(\nu)} \rightarrow 0 \quad \text{as $n \rightarrow \infty$},\]
the first term by dominated convergence and the second term since $\sup_{x \in \R} |\xi_n'(x)| \leq c/n$ for some positive $c > 0$ independent of $n$.
This shows that compactly supported functions are dense in $\mathcal D(C)$ with respect to the graph norm.

Next take $\varphi_n \in C_c^{\infty}(\R)$, again as in the proof of Lemma~\ref{lem:core}. Let $f \in \mathcal D(C)$ with compact support, and define $f_n(x) = (f \star \varphi_n)(x)$. Then every $f_n \in C_c^{\infty}(\R)$ is supported in $\operatorname{supp}(f) + B(1)$ and $f_n \rightarrow f$ in $L^2(\nu)$. Since $\partial_x f_n = (\partial_x f)\star \varphi_n$ and $U'$ is bounded uniformly in $n$ on the support of $(f_n)$, it follows that $\|C f - C f_n\|_{L^2(\nu)} \rightarrow 0$.
\end{proof}

\subsection{Estimates in the proof of Theorem~\ref{thm:spectrum}}
In the following lemma we verify that the constants that appear as integrals in the proof of Theorem~\ref{thm:spectrum} are well defined.

\begin{lemma}
\label{lem:integrals}
Suppose Assumption~\ref{ass:compactness} holds. Let $h \in L^2(\nu)$. For any $\gamma \in \C$, we have that the following mappings are in $L^1(\R)$.
\begin{align} 
\label{eq:integral1} \xi & \mapsto U'(\xi) e^{-2 \gamma \xi - U(\xi)}, \\
\label{eq:integral2} \xi & \mapsto e^{-2 \gamma \xi - U(\xi)},\\
\label{eq:integral3} \xi & \mapsto e^{-\gamma \xi - U(\xi)} h(\xi), \\
\label{eq:integral4} \xi & \mapsto U'(\xi) e^{-\gamma \xi} \int_0^{\xi} e^{\gamma \eta} h(\eta) \, d \eta. 
\end{align}
\end{lemma}

\begin{proof}
Write $\alpha = \Re \gamma$. It suffices to check integrability over $[0,\infty)$, as integrability over $(-\infty,0]$ is analogous. Next, by continuity and where necessary local integrability of $h$ it suffices to check integrability over $[x_0,\infty)$ for some $x_0 > 0$. Let $x_0$ be sufficiently large, such that, for some constant $c > 0$ such that $c + 2 \alpha > 0$, we have that $U'(x) \geq c$ for $x \geq x_0$. We find
\begin{align*}
 \int_{x_0}^{\infty} \left| U'(\xi) e^{-2 \gamma \xi - U(\xi)} \right| \, d \xi & = \int_{x_0}^{\infty} U'(\xi) e^{-2 \alpha \xi - U(\xi)} \, d \xi \\
 & = e^{-U(x_0)} - 2 \alpha \int_0^{\infty} e^{-U(\xi) - 2 \alpha \xi} \, d \xi.
\end{align*}
If $\alpha \geq 0$, we have already succeeded in establishing~\eqref{eq:integral1}. Moreover, for any value of $\alpha \in \R$,
\begin{align*} \int_{x_0}^{\infty} e^{-U(\xi) - 2 \alpha \xi} \, d \xi & \leq \int_{x_0}^{\infty} e^{- U(x_0) - c(\xi-x_0) - 2 \alpha \xi} \, d \xi  < \infty,
\end{align*}
establishing~\eqref{eq:integral1},~\eqref{eq:integral2}.
This yields, for~\eqref{eq:integral3}, that 
\begin{align*}
 \int_0^{\infty} e^{-\alpha \xi - U(\xi)} |h(\xi)| \, d \xi \leq  \|h\|_{L^2(\nu)} \int_0^{\infty} e^{-2 \alpha \xi - U(\xi)} \, d \xi < \infty
\end{align*}
Finally, for~\eqref{eq:integral4}, 
\begin{align*}
 & \int_0^{\infty} |U'(\xi)| e^{-2 \alpha \xi - U(\xi)} \int_0^{\xi} e^{\alpha \eta} |h(\eta)| \, d \eta \, d \xi \\
 & = \int_0^{\infty} \int_{\eta}^{\infty} |U'(\xi)| e^{-2 \alpha \xi - U(\xi)}  \, d \xi  e^{\alpha \eta + U(\eta) } |h(\eta)| \, e^{-U(\eta)} d \eta,
\end{align*}
so it suffices (for this term) to check that 
\[ \eta \mapsto \int_{\eta}^{\infty} |U'(\xi)| e^{-2 \alpha \xi - U(\xi)} \, d \xi \,   e^{\alpha \eta + U(\eta) } \in L^2(\nu).\] 
We have for $\eta \geq x_0$, by partial integration,
\[ \int_{\eta}^{\infty} |U'(\xi)| e^{-2 \alpha \xi - U(\xi)} \, d \xi  e^{\alpha \eta + U(\eta) } = e^{- \alpha \eta} -2 \alpha \int_{\eta}^{\infty} e^{-2 \alpha \xi - U(\xi)} \, d \xi e^{\alpha \eta + U(\eta)}.\]
The first term has already been establised above to belong to $L^2(\nu)$, the second term can be estimated using $U(\xi) - U(\eta) \geq c(\xi-\eta)$ as
\begin{align*}\int_{\eta}^{\infty} e^{-2 \alpha \xi - U(\xi)} \, d \xi e^{\alpha \eta + U(\eta)} \leq \int_{\eta}^{\infty} e^{-2 \alpha \xi - c(\xi - \eta)} \, d \xi e^{\alpha \eta} = \frac{e^{-\alpha \eta}}{2 \alpha + c} < \infty. \end{align*}
\end{proof}

\subsection{Lemmas for the proof of Theorem~\ref{thm:spectralgap}}

The following two lemmas are instrumental in establishing a spectral gap, i.e., Theorem~\ref{thm:spectralgap}.

\begin{lemma}
\label{lem:no-finite-accumulation-point}
Suppose the assumptions of Theorem~\ref{thm:spectrum} are satisfied. Then $\i \beta \in \sigma(L)$ for $\beta \in \R$ if and only if $\beta = 0$.
In particular $\sigma(L)$ does not have an accumulation point at $\i \beta$ for $\beta \in \R$.
\end{lemma}
\begin{proof}
First note that by Assumption~\ref{ass:unimodal}, $U'(0) = 0$ and $x \mapsto U'(x) e^{-U(x)}$ is a probability density function on $[0,\infty)$.
We have
\[ \psi^+(\i \beta) = \int_0^{\infty} U'(\xi) e^{-2 \i \beta \xi - U(\xi)} \, d \xi  = \int_0^{\infty} U'(\xi) [ \cos (2 \beta \xi) - \i \sin(2 \beta \xi)] e^{- U(\xi)} \, d \xi,\] so that, by Jensen's inequality,
\begin{align*}
 |\psi^{+}(\i \beta)|^2 & = \left( \int_0^{\infty} U'(\xi) \cos(2 \beta \xi) e^{-U(\xi)} \, d \xi \right)^2 +  \left( \int_0^{\infty} U'(\xi) \sin(2 \beta \xi) e^{-U(\xi)} \, d \xi \right)^2 \\
 & \leq \int_0^{\infty} U'(\xi) [\cos^2(2 \beta \xi) + \sin^2(2 \beta \xi)] e^{-U(\xi)} \, d \xi = \int_0^{\infty} U'(\xi) e^{-U(\xi)} \, d \xi = 1,
 \end{align*}
with equality if and only if the integrands are constant, i.e., if and only if $\beta = 0$.
Similarly $|\psi^-(\i \beta)| \leq 1$ with equality if and only if $\beta = 0$. It follows that $Z(\i \beta) = 1 -\psi^+(\i \beta) \psi^-(\i \beta) =  0$ if and only if $\beta = 0$.

Now suppose there is $\beta \neq 0$ and a sequence $\gamma_n \in \sigma(L)$ such that $\lim_{n \rightarrow \infty} \gamma_n = \i \beta$. By continuity of $Z$ (Lemma~\ref{lem:psi}), it follows that $Z(\i \beta) = 0$, a contradiction.
Furthermore, since $0$ an eigenvalue, and the spectrum consists of isolated eigenvalues only, 0 can not be an accumulation point either.
\end{proof}

\begin{lemma}
\label{lem:no-infinite-accumulation-point}
Suppose the assumptions of Theorem~\ref{thm:spectrum} are satisfied. For every closed and bounded interval $[a,b] \subset \R$ there is a constant $C > 0$ such that
\[ |\psi^{\pm}(\alpha + \i \beta)| \leq \frac C {|\beta|} \quad \text{for all} \quad \alpha \in [a,b] \quad \text{and} \quad \beta \neq 0.\]
In particular, $\alpha + \i \beta \in \rho(L)$ for $ \alpha \in [a,b]$, $|\beta| > C$.
\end{lemma}

\begin{proof}
By partial integration, and using that by Assumption~\ref{ass:sobolev}, $|U''(\xi)| \leq c_1 + c_2 |U'(\xi)|^2$, for some constants $c_1, c_2 > 0$,
\begin{align*}
 | \psi^+(\alpha + \i \beta) | & = \left| \int_0^{\infty} U'(\xi) e^{-2 \alpha \xi - 2 \i \beta \xi - U(\xi)} \, d \xi \right| \\
 & = \left|  \frac 1 {2 \i \beta} \int_0^{\infty} [ U''(\xi) - 2 \alpha U'(\xi) - (U'(\xi))^2 ] e^{-2 \alpha \xi - 2 \i \beta \xi - U(\xi)} \, d \xi  \right| \\
 & \leq \frac 1 {2 |\beta|} \int_0^{\infty} |U''(\xi) - 2 \alpha U'(\xi) - (U'(\xi))^2| e^{-2 \alpha \xi - U(\xi)} \, d \xi \\
 & \leq \frac 1 {2 |\beta|} \int_0^{\infty} [c_1 +2 \alpha^2 + (c_2 + 3/2) (U'(\xi))^2] e^{-2 \alpha \xi - U(\xi)} \, d \xi,
\end{align*}
and the integral converges by an application of Lemma~\ref{lem:sobolev-classical}. It follows that, for $\alpha \in [a,b]$,
\[| \psi^+(\alpha + \i \beta) | \leq  \frac 1 {2 |\beta|} \int_0^{\infty} [c_1 +2 \max(|a|,|b|)^2 + (c_2 + 3/2) (U'(\xi))^2] e^{-2 a \xi - U(\xi)} \, d \xi =: C^+/|\beta|.\]
The result for $\psi^{-}$ follows analogously, yielding a constant $C^-$, and we may take $C$ as the maximum of the two constants obtained.
Thus for $|\beta| > C$, we have that
\[ |\psi^+(\alpha + \i \beta) \psi^-(\alpha + \i \beta) | \leq C^2/\beta^2 < 1,\]
so that $Z(\alpha + \i \beta) \neq 0$ and therefore $\alpha + \i \beta \in \rho(L)$.
\end{proof}

\subsection{A convolution estimate}
In the proof of Proposition~\ref{prop:resolvent-bound} we will make repeated use of the following lemma.
\begin{lemma}
	\label{lem:convolution}
	Suppose Assumption~\ref{ass:polynomial-tail} is satisfied. Let $\phi : \R \rightarrow \C$ be Lebesgue measurable and $\psi \in L^2(\nu)$. Furthermore assume that
	$\phi(x) =\psi(x) = 0$ for almost every $x < 0$. Let 
	\[ (\phi \star \psi)(x) := \int_{0}^{x} \phi(x-y) \psi(y) \, d y, \quad \text{and} \quad \zeta(x) := e^{U(x)}\int_x^{\infty} \phi(y-x) \psi(y) e^{-U(y)} \, d y,  \quad x \geq 0. \]
	Then
	\[ \| \phi \star \psi\|_{L^2(\nu)} \leq \left( \int_0^{\infty} |\phi(x)| e^{-C/2 - (m/2)|x|^p} \, d x \right) \|\psi\|_{L^2(\nu)},\]
	and
	\[ \| \zeta\|_{L^2(\nu)} \leq \left( \int_{0}^{\infty} |\phi(x)| e^{-C/2 -(m/2)|x|^p}\, d x \right) \| \psi\|_{L^2(\nu)}. \ \]
\end{lemma}

\begin{proof}
	By Lemma~\ref{lem:polynomial-tail} we may assume without loss of generality that Assumption~\ref{ass:polynomial-tail} holds with $M = 0$, i.e., there are constants $C \leq 0$, $m > 0$ and $p > 1$ such that 
	\[U(y) \geq U(x) + C + m|x-y|^p \quad \text{for all $x, y$ for which  $0 \leq x \leq y$ or $y \leq x \leq 0$.}
	\]
	Write $\hat \psi(x) = e^{-U(x)/2} \psi(x)$, so that $\|\hat \psi\|_{L^2(\R)} = \|\psi\|_{L^2(\nu)}$.
	By Assumption~\ref{ass:polynomial-tail},
	\begin{align*}
	e^{-U(x)/2} \int_{0}^{x} \phi(x-y) \psi(y) \, d y
	& = e^{-U(x)/2} \int_{0}^{x} \phi(x-y) \hat \psi(y) e^{U(y)/2} \, d y \\
	& \leq \int_{0}^{x} |\phi(x-y)| e^{- C/2 - (m/2)|y-x|^p}|\hat \psi(y)|\, dy  \\
	& = (|\hat \phi|  \star |\hat \psi|)(x),
	\end{align*}
	where
	\[ \hat \phi(x) = \phi(x) e^{-C/2 - (m/2)|x|^p}.\]
	The first result of the lemma follows now from the convolutional inequality $\| f \,  \star \,g\|_{L^2(\R)} \leq \|f \|_{L^1(\R)} \|g \|_{L^2(\R)}$.
	Next
	\begin{align*} e^{-U(x)/2} \zeta(x) & = e^{U(x)/2} \int_x^{\infty} \phi(y-x) \hat \psi(y) e^{-U(y)/2} \, d y \\
	& \leq \int_x^{\infty} |\phi(y-x)| e^{-C/2 -(m/2) |x-y|^p/2} |\hat \psi(y)|  \, d y \\
	& = (|\tilde \phi| \star |\hat \psi|)(x),
	\end{align*}
	where
	\[ \tilde \phi(x) = \phi(-x) e^{-C/2 -(m/2)|x|^p},\]
	which yields the second statement after applying the convolutional inequality again.
\end{proof}

\end{document}